\documentclass[12pt]{amsart}

\usepackage[hmargin=0.8in,height=8.6in]{geometry}
\usepackage{amssymb,amsthm, times}
\usepackage{delarray,verbatim}
\usepackage{ifpdf}
\ifpdf
\usepackage[pdftex]{graphicx}
 \else
\usepackage[dvips]{graphicx}
 \fi

\usepackage{ifpdf}
\usepackage{color}
\definecolor{webgreen}{rgb}{0,.5,0}
\definecolor{webbrown}{rgb}{.8,0,0}
\definecolor{emphcolor}{rgb}{0.95,0.95,0.95}

\usepackage{hyperref}
\hypersetup{%
          colorlinks=true,
          linkcolor=webbrown,
          filecolor=webbrown,
          citecolor=webgreen,
          breaklinks=true}
\ifpdf \hypersetup{pdftex,
            pdfstartview=FitH, %%Fit, FitB, FitH
            bookmarksopen=true,
            bookmarksnumbered=true
} \else \hypersetup{dvips} \fi

\renewcommand{\S}{\mathcal{S}}

\numberwithin{equation}{section}
\newtheorem{spec}{Specification}[section]

\newtheorem{proposition}{Proposition}[section]

\newtheorem{remark}{Remark}[section]
\newtheorem{lemma}{Lemma}[section]

\newtheorem{assump}{Assumption}[section]

\numberwithin{remark}{section} \numberwithin{proposition}{section}
\numberwithin{corollary}{section}
\newcommand {\R}{\mathbb{R}}

\newcommand {\F}{\mathcal{F}}

\newcommand {\p}{\mathbb{P}}

\newcommand {\E}{\mathbb{E}}

\newcommand{\diff}{{\rm d}}

\newcommand{\lev}{L\'{e}vy }

\title[continuous and smooth fit principle]{On the continuous and smooth fit principle for optimal stopping problems in spectrally negative L\'{e}vy models}
\author[M. Egami]{Masahiko Egami}
\address[M. Egami]{Graduate School of Economics,
Kyoto University, Sakyo-Ku, Kyoto, 606-8501, Japan }
\email{egami@econ.kyoto-u.ac.jp}
\urladdr{http://www.econ.kyoto-u.ac.jp/{\textasciitilde}egami/}
\thanks{This version: \today.  \\ M.\ Egami is in part supported
by Grant-in-Aid for Scientific Research (B) No.\ 22330098 and (C)
No.\ 20530340, Japan Society for the Promotion of Science.  K.\ Yamazaki is in part supported by Grant-in-Aid for Young Scientists
(B) No.\ 22710143, the Ministry of Education, Culture, Sports,
Science and Technology, and by Grant-in-Aid for Scientific Research (B) No.\  2271014, Japan Society for the Promotion of Science. The authors thank the anonymous referee for her/his thorough reviews and insightful comments that help improve the presentation of this paper.}
\author[K. Yamazaki]{Kazutoshi Yamazaki }
\address[K. Yamazaki]{Department of Mathematics,
Faculty of Engineering Science, Kansai University, Suita-shi, Osaka 564-8680, Japan}
\email{kyamazak@kansai-u.ac.jp}
\urladdr{https://sites.google.com/site/kyamazak/}
\date{}
\begin{document}
\begin{abstract}
We consider a class of infinite-time horizon optimal stopping problems for spectrally negative \lev processes.  
Focusing on strategies of threshold type, we write explicit expressions for the corresponding expected payoff via the scale function, and further pursue optimal candidate threshold levels. We obtain and show the equivalence of the continuous/smooth fit condition and the first-order condition for maximization over threshold levels.
%
% show the equivalence between the continuous/smooth fit condition and the first-order condition for maximization over threshold levels. 
%We achieve the first-order condition that maximizes it over all threshold levels as well as the continuous and smooth fit conditions and show their equivalence. 
% Unlike the continuous linear diffusion counterpart, their characterizations depend on the overshoot at the first down-crossing time and hence the results based on continuity assumption no longer hold.  We alleviate this issue by using the potential measure written in terms of the scale function, and further pursue continuous and smooth fit conditions in order to identify candidate threshold levels.
%This together with problem-specific information about the payoff function can prove optimality over all stopping times.
%This paper provides with a practical tool in solving optimal
%stopping problems for a general spectrally negative \lev process.  
 As examples of its applications, we give a short proof of the McKean optimal stopping problem (perpetual American put option) and solve an extension to Egami
and Yamazaki \cite{Egami-Yamazaki-2010-1}.

\end{abstract}

\maketitle \noindent \small{\textbf{Key words:} Optimal stopping;
 Spectrally negative \lev processes; Scale functions; Continuous and smooth fit\\
%\noindent JEL Classification: G32, D81, C61 \\
 }\\

\section{Introduction}
Optimal stopping problems arise in various areas ranging from the
classical sequential testing/change-point detection problems to
applications in finance.   Although all formulations reduce to the problem
of maximizing/minimizing the expected payoff over a set of
stopping times, the solution methods are mostly problem-specific;
they depend significantly on the underlying process, payoff function and time-horizon.  This paper pursues a common tool for the class
of \emph{infinite-time horizon} optimal stopping problems for \emph{spectrally
negative \lev processes}, or \lev processes with only negative jumps.

By extending the classical continuous diffusion model to the \lev model, one can achieve richer and more realistic models. In mathematical finance, the continuity of paths is empirically rejected and cannot explain, for example, the volatility smile and non-zero credit spreads for short-maturity corporate bonds.  These issues can often be alleviated by introducing jumps; see, e.g.\  \cite{Chen_Kou_2009, Kou_Wang_2004}.   Recently, we saw a significant progress in the theory of optimal stopping for \lev  processes and other jump processes.  The fluctuation theory, in particular, has been playing a key role in characterizing efficiently the value function and the optimal stopping time; see \cite{Boyarchenko_Levendorskii_2007, Christensen_2013, Mordecki_Salminen_2007,Surya_2007} among others. %  These results can be applied directly to derive classical results such as the perpetual American and Russian options.  

In this paper, we revisit the optimal stopping problem for a general spectrally negative \lev process, and pursue a solution in a rather straightforward way.  Despite the aforementioned results, existing results  under  \lev processes are still significantly more limited than the one-dimensional diffusion case as in \cite{alvarez2, Beibel_Lerche_2000, Christensen_2011,  DK2003, dynkin,  peskir-shiryaev}.  Without the continuity of paths, the process can jump over a given threshold. For a general payoff function, one naturally needs to take care of the overshoot distribution, which is generally a big hurdle that typically makes the problem intractable.   However, thanks to the recent advances in the fluctuation theory of spectrally negative \lev processes (see \cite{Bertoin_1996, Kyprianou_2006}), this can be handled  by using the so-called \emph{scale function}.

The objective of this paper is to pursue, with the help of the scale function, a common technique for the class of optimal stopping problems for spectrally negative \lev
processes.  Focusing on the first time it down-crosses a fixed threshold, we express the corresponding expected payoff in terms of the scale function.  This semi-explicit form enables us to differentiate and take limits thanks to the smoothness and asymptotic properties of the scale function as obtained, for example, in \cite{Chan_2009, Kyprianou_2006}.  By differentiating the expected payoff with respect to the threshold level, we obtain the \emph{first-order condition} as well as the candidate optimal level that makes it vanish. We also obtain the \emph{continuous/smooth fit condition} when the process is of bounded variation or when it contains a diffusion component.  These conditions are in fact equivalent and can be obtained generally under mild conditions.

% we obtain the continuous and smooth fit conditions that are commonly used to identify candidate optimal threshold levels.
%
%
%To be more precise, we achieve the followings.
%\begin{enumerate}
%\item 
%\item 
%\item We show the equivalence between (2) and (3).
%\item We study what additional conditions are necessary in order to prove the optimality over all stopping times.  As an example, we extend Egami and Yamazaki.
%\end{enumerate}

%Solutions to infinite-time horizon problems are typically given by first down-crossing times.  Moreover, in some special cases such as the pricing of American put options (e.g.\ Mordecki \cite{Mordecki_2002}) and the problem of determining optimal bankruptcy levels (e.g.\ Kyprianou and Surya \cite{Kyprianou_Surya_2007}), the optimality of down-crossing time is already proven regardless of the type of underlying \lev processes.  We shall focus on the first down-crossing time and express the corresponding expected payoff in terms of the scale function.  This semi-explicit form enables us to differentiate and take limits thanks to the smoothness and asymptotic properties of the scale function as obtained, for example, in Chan et al.\ \cite{Chan_2009} and Kyprianou
%\cite{Kyprianou_2006}; we obtain the continuous and smooth fit conditions that are commonly used to identify candidate optimal threshold levels.

The spectrally negative \lev model has been drawing much attention recently as a generalization of the classical Black-Scholes model in mathematical finance and also as a generalization of the  Cram\'er-Lundberg model in insurance.  A  number of authors have succeeded in extending the classical results to the spectrally negative \lev model by way of scale functions.  We refer the reader to \cite{Baurdoux2008,Baurdoux2009}  for stochastic games,  \cite{Avram_et_al_2007, Bayraktar_2012, Bayraktar_2013, Kyprianou_Palmowski_2007, Loeffen_2008}  for the optimal dividend problem, \cite{alili-kyp, avram-et-al-2004} for American and Russian options,  \cite{Leung_Yamazaki_2011, Kyprianou_Surya_2007, Leung_Yamazaki_2010} for credit risk and \cite{Yamazaki_2013} for inventory models.  In particular, Egami and Yamazaki \cite{Egami-Yamazaki-2010-1} modeled and obtained the optimal timing of capital reinforcement.  As an application of the results obtained in this paper, we give a short proof of the McKean optimal stopping (perpetual American put option) problem with additional running rewards, as well as an extension and its analytical solution to \cite{Egami-Yamazaki-2010-1}.

The rest of the paper is organized as follows.  In Section
\ref{sec:problem}, we review the optimal stopping problem for
spectrally negative \lev processes, and then
express the expected value corresponding to the first down-crossing time in terms of
the scale function.  In Section \ref{sec:fit}, we obtain the first-order condition as well as the continuous/smooth fit condition and show their equivalence. In Section
\ref{section_extension}, we solve the McKean optimal stopping problem and an extension to \cite{Egami-Yamazaki-2010-1}.  We conclude the paper in  Section \ref{section_conclusion}.
%Some long proofs are postponed to the appendix.

\section{The Optimal Stopping Problem for Spectrally Negative \lev Processes}\label{sec:problem}
Let $(\Omega, \F, \p)$ be a probability space hosting a
spectrally negative \lev process $X=\{X_t: t\ge 0\}$
characterized uniquely by the \emph{Laplace exponent}
\begin{align}
\psi(\beta) := \E^0 \left[ e^{\beta X_1}\right] = c \beta
+\frac{1}{2}\sigma^2 \beta^2 +\int_{( 0,\infty)}(e^{-\beta
z}-1+\beta z 1_{\{0 <z<1\}})\,\Pi(\diff z), \quad  {\beta \in
\mathbb{R}}, \label{laplace_exp}
\end{align}
where $c \in\R$, $\sigma\geq 0$ and $\Pi$ is a measure  on
$(0,\infty)$ such that 
\begin{align}
\int_{(0,\infty)} (1  \wedge z^2) \Pi(
\diff z)<\infty. \label{integrability_levy_measure}
\end{align}
Here and throughout the paper, $\p^x$ is the conditional probability where $X_0 = x \in \R$ and $\E^x$ is its expectation (also $\p \equiv \p^0$ and $\E \equiv \E^0$).  It is well-known that $\psi$ is zero at the origin, convex on $\R_+$
and has a right-continuous inverse:
\[
\Phi(q) :=\sup\{\lambda \geq 0: \psi(\lambda)=q\}, \quad q\ge 0.
\]
In particular, when
\begin{align}
\int_{(0,\infty)} (1 \wedge z)\, \Pi(\diff z) < \infty,
\label{cond_bounded_variation}
\end{align}
we can rewrite
\begin{align*}
\psi(\beta) =\mu  \beta+  \frac{1}{2}\sigma^2 \beta^2 + \int_{(
0,\infty)}(e^{-\beta z}-1)\,\Pi(\diff z), \quad \beta \in \mathbb{R}
\end{align*}
where
\begin{align*}
\mu := c + \int_{( 0,1)}z\, \Pi(\diff z). %\label{def_mu}
\end{align*}
The process has paths of bounded variation if and only if $\sigma =
0$ and \eqref{cond_bounded_variation} holds. It is also assumed that
$X$ is not a negative subordinator (decreasing a.s.). Namely, we
require $\mu$ to be strictly positive if $\sigma = 0$ and \eqref{cond_bounded_variation} holds.

Let $\mathbb{F} := (\F_t)_{t \geq 0}$ be the filtration generated by $X$ and $\S$ a set of $\mathbb{F}$-stopping times.  We shall consider a general optimal stopping problem of the form:
\begin{align}
u(x) := \sup_{\tau \in \S} \E^x \left[ e^{-q \tau}
g(X_{\tau}) 1_{\{ \tau < \infty \}}+  \int_0^{\tau} e^{-qt} h(X_t) \diff
t \right], \quad x \in \R \label{general_problem}
\end{align}
for some discount factor $q>0$ and locally-bounded measurable functions $g, h: \R \mapsto \R$ which represent, respectively,  the payoff received at a given stopping time
$\tau$ and the running reward up to $\tau$.  It is assumed here that the expectation is well-defined; we shall give further assumptions on $g$ and $h$ in order to guarantee that it is indeed so.  See Assumption \ref{assump_h} and the assumptions in Lemma \ref{lemma_lipschitz} below.
%\blue{[need to define $\F$]}.   %The problem in the precautionary paper is a special case
%where $g(x) = -1_{\{x \leq 0\}}$, $h$ is nonnegative and increasing,
%and $\S$ is the set of all stopping times smaller than or equal to
%$\theta := \inf \left\{ t \geq 0: X_t \leq 0 \right\}$. Here we
%consider a general case.

Typically, the optimal stopping time is given by
the \emph{first down-crossing time} of the form
\begin{align}
\tau_A := \inf \left\{ t > 0: X_t \leq A \right\}, \quad A \in \R, \label{def_tau_A}
\end{align}
with $\inf \emptyset = \infty$. Let us denote the corresponding expected payoff by
\begin{align*}
u_A(x) &:= \E^x \left[ e^{-q \tau_A} g(X_{\tau_A})  1_{\{ \tau_A < \infty \}} + \int_0^{\tau_{A}} e^{-qt} h(X_t) \diff t \right], \quad x, A \in \R, %\label{u_A}
\end{align*}
which can be decomposed into
\begin{align*}
u_A(x) &= \left\{ \begin{array}{ll} \Gamma_1(x;A) +  \Gamma_2(x;A)  +  \Gamma_3(x;A), &  x > A, \\ g(x), & x \leq A, \end{array} \right.
\end{align*}
where, for every $x > A$,
\begin{align}
\begin{split}
\Gamma_1(x;A) &:=g(A) \E^x \left[ e^{-q \tau_A} \right], \\
\Gamma_2(x;A) &:= \E^x \left[ e^{-q \tau_A} (g(X_{\tau_A})-g(A)) 1_{\{X_{\tau_A} < A, \, \tau_A < \infty \}}\right], \\
\Gamma_3(x;A) &:= \E^{x}\left[ \int_0^{\tau_{A}} e^{-qt} h(X_t) \diff t \right].
\end{split} \label{definition_gamma}
\end{align}
Shortly below, we express each term via the scale function.

%Consequently,
%$\Gamma_1(x;A) = 0$ uniformly in $x$ and $A$ if $\sigma =0$. 

\begin{remark}
This paper does not consider the \emph{first up-crossing time} defined by $\tau_B^+ := \inf \left\{ t > 0: X_t \geq B \right\}$ because, for the spectrally negative \lev case, the process always creeps upward ($g(X_{\tau_B^+}) = g(B)$ a.s.\ on $\{\tau_B^+ < \infty \}$), and the expression of the expected value is much simplified.  We focus  on a more interesting and challenging case where the optimal stopping time is conjectured to be a first down-crossing time.  We refer the reader to, among others, \cite{Christensen_2013} and \cite{Salminen_2011} for related problems under a more general Markov process. 
\end{remark}

\begin{remark}
It is possible to reduce the problem so that the running reward part is zero (i.e.\ $h\equiv 0$); see, e.g., \cite{Cisse_2012} for the reduction technique.  However, we decide to solve the problem in the current form because the running reward part can be handled more easily than the stopping payoff part, especially  for the verification of optimality.   In our examples in Section \ref{section_extension}, the optimality holds under a mild monotonicity on the function $h$; this gives a more intuitive explanation about the problems and their optimal solutions. 
\end{remark}

\subsection{Scale functions} \label{subsec:scale_functions}
In this subsection, we give a brief review on the scale function that will be needed for our analysis.
For a comprehensive account of the scale function, see
\cite{Bertoin_1996,Bertoin_1997, Kyprianou_2006, Kyprianou_Surya_2007}. See \cite{Egami_Yamazaki_2010_2, Kuznetsov_2012, Surya_2008} for numerical methods for computing the
scale function.

Associated with every spectrally negative \lev process, there exists a ($q$-)scale function
%
%Associated with every spectrally negative \lev process, there exists
%a \emph{(q-)scale function}
\begin{align*}
W^{(q)}: \R \mapsto \R; \quad q\ge 0,
\end{align*}
that is continuous and strictly increasing on $[0,\infty)$ and is
uniquely determined by
\begin{align*}%\label{eq:scale}
\int_0^\infty e^{-\beta x} W^{(q)}(x) \diff x = \frac 1
{\psi(\beta)-q}, \qquad \beta > \Phi(q).
\end{align*}
%where
%\begin{align*}
%\Phi(q) = \sup \left\{ \lambda > 0: \psi(\lambda) = q \right\}, \quad q \geq 0.
%\end{align*}

Fix $a > x > 0$.  If $\tau_a^+$ is the first time the process goes above $a$ and $\tau_0$ is the first time it goes below zero as a special
case of \eqref{def_tau_A}, then we have
\begin{align*}
\E^x \left[ e^{-q \tau_a^+} 1_{\left\{ \tau_a^+ < \tau_0, \, \tau_a^+ < \infty
\right\}}\right] = \frac {W^{(q)}(x)}  {W^{(q)}(a)} \quad
\textrm{and}  \quad \E^x \left[ e^{-q  \tau_0} 1_{\left\{ \tau_a^+ >
 \tau_0, \, \tau_0 < \infty \right\}}\right] = Z^{(q)}(x) - Z^{(q)}(a) \frac {W^{(q)}(x)}
{W^{(q)}(a)},
\end{align*}
where
\begin{align*}%\label{eq:Z-q-function}
Z^{(q)}(x) := 1 + q \int_0^x W^{(q)}(y) \diff y, \quad x \in \R.
\end{align*}
Here we have
\begin{equation}\label{eq:at-zero}
W^{(q)}(x)=0 \quad\text{on} \quad  (-\infty,0)\quad\text{ and}\quad
Z^{(q)}(x)=1 \quad\text{on}\quad (-\infty,0].
\end{equation}
We also have
\begin{align}
\E^x \left[ e^{-q  \tau_0} \right] = Z^{(q)}(x) - \frac q {\Phi(q)} W^{(q)}(x), \quad x > 0. \label{laplace_tau_0}
\end{align}

 In particular, $W^{(q)}$ is continuously differentiable on $(0,\infty)$ if $\Pi$ does not have atoms and $W^{(q)}$ is twice-differentiable on $(0,\infty)$ if $\sigma > 0$; see, e.g., \cite{Chan_2009}.  Throughout this paper, we assume the former.
\begin{assump}
We assume that $\Pi$ does not have atoms.
\end{assump}

Fix $q > 0$.  The scale function increases exponentially; 
\begin{align}
W^{(q)} (x) \sim \frac {e^{\Phi(q) x}} {\psi'(\Phi(q))} \quad
\textrm{as } \; x \uparrow \infty.
\label{scale_function_asymptotic}
\end{align}
There exists a (scaled) version of the scale function $ W_{\Phi(q)}
= \{ W_{\Phi(q)} (x); x \in \R \}$ that satisfies
\begin{align}
W_{\Phi(q)} (x) = e^{-\Phi(q) x} W^{(q)} (x), \quad x \in \R \label{W_scaled}
\end{align}
and
\begin{align*}
\int_0^\infty e^{-\beta x} W_{\Phi(q)} (x) \diff x &= \frac 1
{\psi(\beta+\Phi(q))-q}, \quad \beta > 0.
\end{align*}
Moreover $W_{\Phi(q)} (x)$ is increasing, and as is clear from
\eqref{scale_function_asymptotic},
\begin{align}
W_{\Phi(q)} (x) \uparrow \frac 1 {\psi'(\Phi(q))} \quad \textrm{as }
\; x \uparrow \infty. \label{scale_function_asymptotic_version}
\end{align}

Regarding its behavior in the neighborhood of zero, it is known that
\begin{align}
W^{(q)} (0) = \left\{ \begin{array}{ll} 0, & \textrm{unbounded
variation} \\ \frac 1 {\mu}, & \textrm{bounded variation}
\end{array} \right\} \quad \textrm{and} \quad W^{(q)'} (0+) =
\left\{ \begin{array}{ll}  \frac 2 {\sigma^2}, & \sigma > 0 \\
\infty, & \sigma = 0 \; \textrm{and} \; \Pi(0,\infty) = \infty \\
\frac {q + \Pi(0,\infty)} {\mu^2}, & \textrm{compound Poisson}
\end{array} \right\}; \label{at_zero}
\end{align}
see Lemmas 4.3 and 4.4 of 
\cite{Kyprianou_Surya_2007}.

%In order to solve our problem for the general form of $h$, we express, in terms of the scale function, the regret function for every fixed threshold level  $A\in [0, x]$.

\subsection{Expressing the expected payoff using the scale function}
We now express \eqref{definition_gamma} in terms of the scale function.  For the rest of the paper, because $q > 0$, we must have $\Phi(q) > 0$.

First, the following is immediate by \eqref{laplace_tau_0}.
\begin{lemma} \label{lemma_gamma_1}
For every $x > A$, we have 
\begin{align*}
\Gamma_1(x;A) =  g(A) \left[ Z^{(q)}(x-A) - \frac q {\Phi(q)} W^{(q)}(x-A) \right].
\end{align*}
\end{lemma}
For $\Gamma_2$ and $\Gamma_3$, we use the potential measure written in terms of the scale function.  
%By using Theorem 1 of \cite{Bertoin_1997} (see also
%\cite{Emery_1973, Suprun_1976}),  we have, for every
%$B \in \mathcal{B}(\mathbb{R})$ and $a > x > A$,
%\begin{align*}
%\E^x \left[ \int_0^{\tau_{A} \wedge \tau^+_a} e^{-qt} 1_{\left\{ X_t \in B \right\}} \diff t\right] &= \int_{B \cap [A,\infty)} \left[ \frac {W^{(q)}(x-A) W^{(q)} (a-y)} {W^{(q)}(a-A)} - W^{(q)} (x-y) \right] \diff y.
%\end{align*}
%By taking $a \uparrow \infty$ via the dominated convergence theorem, we can obtain $\Gamma_3(x;A)$ in \eqref{definition_gamma}. 
For the problem to be well-defined, we assume throughout the paper the following so that $\Gamma_3$ is finite.  For a complete proof of Lemma \ref{lemma_gamma_3} below, see \cite{Egami-Yamazaki-2010-1}.
\begin{assump} \label{assump_h} We assume that $\int_0^\infty e^{-\Phi(q)y} | h(y) | \diff y < \infty$.
\end{assump}
\begin{lemma} \label{lemma_gamma_3}For all $x > A$, we have
\begin{align*}
\Gamma_3(x;A) =
W^{(q)} (x-A) \int_0^\infty e^{- \Phi(q) y}   h(y+A) \diff y -
\int_A^x W^{(q)} (x-y) h(y) \diff y.
\end{align*}
%Notice here that $\Gamma_3(x;A) = 0$ for every $x < A$.
\end{lemma}

For $\Gamma_2$, we first define, for every $A \in \R$,
\begin{align}
\begin{split}
\rho_{g,A}^{(q)} &:= \int_0^\infty \Pi (\diff u)   \int_0^{u} e^{-\Phi(q) z} (g(z+A-u)-g(A)) \diff z  \\ &\equiv \int_0^\infty \Pi (\diff u)   \int_A^{u+A} e^{-\Phi(q) (y-A)} (g(y-u)-g(A)) \diff y, \\
 \overline{\rho}_{g,A}^{(q)} &:= \int_0^\infty \Pi (\diff u)  \int_0^{u} e^{-\Phi(q) z} |g(z+A-u)-g(A)| \diff z \\ &\equiv\int_0^\infty \Pi (\diff u)   \int_A^{u+A} e^{-\Phi(q) (y-A)} |g(y-u)-g(A)| \diff y.
\end{split} \label{def_rho}
\end{align}
%and assume the following.
%\begin{assump} \label{assumption_fit}
%For a given $A \in \R$, we assume $\overline{\rho}_{g,A}^{(q)} < \infty$.
%\end{assump}
%These are finite unless $g$ explodes rapidly (due to the decay factor $e^{-\Phi(q) x}$).  
\begin{lemma} \label{lemma_lipschitz}
Fix $A \in \R$.  Suppose
\begin{enumerate}
\item $g$ is $C^2$ in a neighborhood of $A$ and
\item $g$ satisfies 
\begin{align}
\int_1^\infty  \Pi (\diff u)  \max_{A- u \leq \zeta \leq A}|g(\zeta)-g(A)|  < \infty, \label{integrability_of_g} 
\end{align}
\end{enumerate}
then $\overline{\rho}_{g,A}^{(q)} < \infty$.
\end{lemma}
\begin{proof}
See Appendix \ref{proof_lemma_lipschitz}.
\end{proof}
For every $x > A$, we also define
\begin{align*}
\varphi_{g,A}^{(q)}(x) &:=\int_0^\infty \Pi (\diff u)  \int_0^{u \wedge (x-A)} W^{(q)} (x-z-A) (g(z+A-u)-g(A)) \diff z, \\
\overline{\varphi}_{g,A}^{(q)}(x) &:=\int_0^\infty \Pi (\diff u) \int_0^{u \wedge (x-A)} W^{(q)} (x-z-A) |g(z+A-u)-g(A)| \diff z.
%\varphi_A^{(q)}(x) &:= \int_0^\infty \Pi (\diff u) \int_0^{u \wedge (x-A)} W^{(q)} (x-z-A) \diff z.
\end{align*}
By \eqref{W_scaled} and \eqref{scale_function_asymptotic_version},
\begin{align} \label{bound_rho_by_rho}
\begin{split}
\overline{\varphi}_{g,A}^{(q)}(x) 
%&= \int_0^\infty \Pi (\diff u)   \int_0^{u \wedge (x-A)} e^{\Phi(q) (x-z-A)}W_{\Phi(q)} (x-z-A) |g(z+A-u) - g(A)| \diff z  \\
&= e^{\Phi(q) (x-A)} \int_0^\infty \Pi (\diff u)  \int_0^{u  \wedge (x-A)} e^{-\Phi(q) z}W_{\Phi(q)} (x-z-A) |g(z+A-u) - g(A) |\diff z  \\ &\leq e^{\Phi(q) (x-A)} \frac {\overline{\rho}_{g,A}^{(q)}} {\psi'(\Phi(q))},
\end{split}
\end{align}
and hence the finiteness of $\overline{\rho}_{g,A}^{(q)}$ also implies that of $\overline{\varphi}_{g,A}^{(q)}(x)$ for any $x > A$.

Using these notations, Lemma \ref{lemma_gamma_3} together with the compensation formula shows the following.

\begin{lemma}\label{lemma_gamma_2}
If (1)-(2) of Lemma \ref{lemma_lipschitz} hold for a given $A \in \R$, then
%\begin{multline*}
%\Gamma_2(x;A) =   \int_0^\infty \Pi (\diff u) \left[  W^{(q)} (x-A) \int_0^u e^{- \Phi(q) y}   (g(y+A-u)-g(A)) \diff y \right. \\ \left.
%- \int_0^{u \wedge (x-A)} W^{(q)} (x-z-A) (g(z+A-u)-g(A)) \diff z \right]. %\label{gamma_2}
%\end{multline*}
\begin{align}
\Gamma_2(x;A) &=  W^{(q)} (x-A) \rho_{g,A}^{(q)} - \varphi_{g,A}^{(q)}(x), \quad x > A. \label{gamma_2_split}
%\Gamma_1(x;A)
%&=  g(A) \left[ Z^{(q)}(x-A) - \frac q {\Phi(q)} W^{(q)}(x-A) \right]. \label{gamma_1_tilde_split}
\end{align}
%Notice here that $\Gamma_2(x;A) = 0$ for every $x < A$.
\end{lemma}
\begin{proof}%[Proof of Lemma \ref{lemma_gamma_2}]
Let  $N(\cdot, \cdot)$ be the Poisson random measure associated with the jumps of
$-X$ and $\underline{X}_t := \min_{0 \leq s \leq t} X_s$ for all $t
\geq 0$.  We also let $x_\pm = \max (\pm x, 0)$ for any $x \in \R$. By the compensation formula (see, e.g.,
\cite{Kyprianou_2006}), 
\begin{align*}
&\E^x \left[ e^{-q \tau_A} (g(X_{\tau_A})-g(A))_+ 1_{\{X_{\tau_A} < A, \, \tau_A < \infty \}}\right] \\
&= \E^x \left[ \int_0^\infty  \int_0^\infty N(\diff t, \diff u) e^{ -q t} (g(X_{t-}-u)-g(A))_+1_{\{X_{t-} - u \leq A, \;  \underline{X}_{t-} > A \}}\right] \\
&= \E^x \left[ \int_0^\infty  e^{ -q t} \diff t \int_0^\infty \Pi (\diff u) (g(X_{t-}-u)-g(A))_+ 1_{\{X_{t-} - u \leq A, \;  \underline{X}_{t-} > A \}}\right] \\
&=  \int_0^\infty \Pi (\diff u) \E^x \left[  \int_0^\infty  e^{ -q t}   (g(X_{t-}-u)-g(A))_+1_{\{X_{t-} - u \leq A, \;  \underline{X}_{t-} > A \}} \diff t \right] \\
&=  \int_0^\infty \Pi (\diff u) \E^x \left[  \int_0^{\tau_A}  e^{ -q t}   (g(X_{t-}-u)-g(A))_+1_{\{X_{t-} \leq A + u\}} \diff t \right].
\end{align*}
By setting $h(y) \equiv (g(y-u) - g(A))_+1_{\{y \leq A+u\}}$ or equivalently $h(y+A) \equiv (g(y+A -
u)-g(A))_+1_{\{y \leq u\}}$ in Lemma \ref{lemma_gamma_3}, 
\begin{align*}
&\E^x \left[  \int_0^{\tau_A}  e^{ -q t}   (g(X_{t-}-u)-g(A))_+ 1_{\{X_{t-} \leq A + u\}} \diff t \right] \\ &= W^{(q)} (x-A) \int_0^u e^{- \Phi(q) y}   (g(y+A-u) - g(A))_+\diff y -
\int_A^x W^{(q)} (x-y) (g(y-u)-g(A))_+ 1_{\{y \leq A+u\}} \diff y \\
&= W^{(q)} (x-A) \int_0^u e^{- \Phi(q) y}   (g(y+A-u)-g(A))_+\diff y -
\int_0^{u \wedge (x-A)} W^{(q)} (x-z-A) (g(z+A-u)-g(A))_+ \diff z.
\end{align*}
By substituting this, we have 
\begin{align*}
&\E^x \left[ e^{-q \tau_A} (g(X_{\tau_A})-g(A))_+ 1_{\{X_{\tau_A} < A, \, \tau_A < \infty \}}\right] \\ &=   \int_0^\infty \Pi (\diff u) \left[  W^{(q)} (x-A) \int_0^u e^{- \Phi(q) y}   (g(y+A-u)-g(A))_+ \diff y \right. \\ &\qquad \left.
- \int_0^{u \wedge (x-A)} W^{(q)} (x-z-A) (g(z+A-u)-g(A))_+ \diff z \right],
\end{align*}
which is finite by Lemma \ref{lemma_lipschitz} and  \eqref{bound_rho_by_rho}.  Similarly, we can obtain $\E^x \left[ e^{-q \tau_A} (g(X_{\tau_A})-g(A))_- 1_{\{X_{\tau_A} < A, \, \tau_A < \infty \}}\right]$ and \eqref{gamma_2_split} is immediate by taking the difference.
\end{proof}

In view of \eqref{gamma_2_split} above, we can also write
\begin{align}
\begin{split}
W^{(q)} (x-A) \rho_{g,A}^{(q)}  &=   W_{\Phi(q)} (x-A) e^{\Phi(q) x}\int_0^\infty \Pi (\diff u)\int_A^{u+A} e^{- \Phi(q) y}   (g(y-u)-g(A)) \diff y,  \\ \varphi_{g,A}^{(q)}(x) &= \int_0^\infty \Pi (\diff u)  \int_A^{(u+A) \wedge x} W^{(q)} (x-z) (g(z-u)-g(A)) \diff z.
\end{split} \label{gamma_2_rewrite}
\end{align}

%\section{Change of Measure}
%\begin{align*}
%g(x+\Delta x) = g(x) + \Delta x g'(x) + \Delta x^2 g''(x) + \cdots
%\end{align*}
%
%
%Suppose $g(x) = c x$.  Then
%\begin{align*}
%g(A+\Delta A) - g(A) =  c  \Delta A 
%\end{align*}
%For every fixed $A$, suppose $g(x) = c x -  c A e^{x/A}$.  Then
%\begin{align*}
%g(A+\Delta A) - g(A) = c  \Delta A - c \Delta A + o(\Delta A^2) = o(\Delta A^2).
%\end{align*}
%Matching $2 c A = d e^{-d A}$.
%
%Suppose $g(x) = c x^2$.  Then
%\begin{align*}
%g(A+\Delta A) - g(A) = 2 c A  \Delta A + 2c (\Delta A)^2 
%\end{align*}
%Suppose $g(x) = c x^2 - a e^{b x}$.  Then
%\begin{align*}
%g(A+\Delta A) - g(A) = 2 c A  \Delta A - ab e^{b A} \Delta A.
%\end{align*}
%Setting $b = 1/A$ and $a = 2 A^2$
%
%Suppose $g(x) = e^{c x}$, then
%\begin{align*}
%\E^x g(X_t) = \E^x \exp(c X_t) = e^{c x} \E^0 \exp(c X_t)
%\end{align*}

\section{First-Order Condition and Continuous and Smooth fit}\label{sec:fit}
The most common way of choosing the candidate threshold level is via the continuous and smooth fit principle.  Define
\begin{align*}
u_A(A+) := \lim_{x \downarrow A}u_A(x) \quad \textrm{and} \quad u_A'(A+) := \lim_{x \downarrow A}u_A'(x), \quad A \in \R, % \label{u_limits}
\end{align*}
if these limits exist.  The continuous and smooth fit chooses $A$ such that $u_A(A+) = g(A)$ and $u_A'(A+) = g'(A)$, respectively.  Alternatively, one can differentiate $u_A$ with respect to $A$ and obtain the first-order condition.
%
%This section studies
%\begin{enumerate}
%\item the first-order condition $\partial u_A(x) / \partial A = 0$, and
%\item continuous and smooth fit conditions $u_A(A+) = 0$ and $u_A'(A+) = 0$,
%\end{enumerate}
%where

In this section, we pursue the candidate threshold level $A^*$ in both ways. We first obtain, for a general case, the first-derivative $\partial u_A(x) / \partial A$ and $A$ that makes it vanish, and then the continuous fit condition 
for the case $X$ is of bounded variation and the smooth fit condition for the case $X$ has a diffusion component ($\sigma > 0$).  We further discuss the equivalence of these conditions and how to obtain optimal strategies.

\subsection{First-order condition}
We shall obtain $\partial u_A(x)/\partial A$ for $x > A$.
Let
\begin{align*}%\label{eq:Phi-A}
\Psi(A) :=    -\frac q {\Phi(q)} g(A) + \rho_{g,A}^{(q)}  + \int_0^\infty e^{-\Phi(q) y} h(y+A)\diff y,
\quad A \in \R.
\end{align*}

\begin{proposition}[Derivative of $u_A$ with respect to $A$]
 \label{lemma_derivative_A}
For given $x > A$, suppose (1)-(2) of Lemma \ref{lemma_lipschitz} hold and
\begin{align}
%&\int_1^\infty \Pi (\diff u)  \max_{y \in[A-u, A]}|g(y)-g(A)|  < \infty, \\
&\int_1^\infty \Pi(\diff u) \max_{0 \leq \xi \leq \delta} |g(A+\xi) - g(A+\xi-u)| < \infty, \label{integrability_sup}
\end{align}
for some $\delta > 0$. Then, we have
\begin{align*}
\frac \partial {\partial A} u_A(x) = -\Theta^{(q)}(x-A)  \Big( \Psi(A) - \frac {\sigma^2} 2 g'(A) \Big),
\end{align*}
where
\begin{align*}
\Theta^{(q)}(y) := e^{\Phi(q) y} W_{\Phi(q)}' (y), \quad y > 0.
\end{align*}
%In particular, when $\sigma = 0$ (for both bounded and unbounded cases)
%\begin{align*}
%\frac \partial {\partial A} u_A(x) = -W_{\Phi(q)}' (x-A) e^{\Phi(q) (x-A)} \Psi(A).
%\end{align*}

%(1) When $\sigma = 0$, we have
%\begin{multline*}
%\frac \partial {\partial A} u_A(x) =     -W_{\Phi(q)}' (x-A) e^{\Phi(q) (x-A)} \Psi(A)\\
%+ g(A) e^{\Phi(q) (x-A)} \int_0^\infty \Pi (\diff u) e^{-\Phi(q) u} \left[  W_{\Phi(q)}(x-A)      - W_{\Phi(q)}(x-u-A)  \right] .
%\end{multline*}
%(2) When $\sigma > 0$, we have
%\begin{align*}
%\frac \partial {\partial A} u_A(x) = g'(A)Q(x;
%A)-e^{\Phi(q)(x-A)}W_{\Phi(q)}^{'}(x-A)\left\{\Psi(A)-{g(A)}
%\left(\mu + \frac 1 2 \sigma^2 \Phi(q) \right) \right\}.
%\end{align*}
%(3) In Case 3, we have
%\begin{align*}
%\frac \partial {\partial A} u_A(x) =     -W_{\Phi(q)}' (x-A) e^{\Phi(q) (x-A)} \Psi(A).
%\end{align*}
\end{proposition}
Because $W_{\Phi(q)}$ is increasing, $\Theta^{(q)}$ is positive (see also \cite{Kyprianou_Surya_2007} for an interpretation of $\Theta^{(q)}$ as the resolvent measure of the ascending ladder height process of $X$) and hence
\begin{align}
\Psi(A) - \frac {\sigma^2} 2 g'(A) \leq (\geq) 0 \Longrightarrow \frac \partial {\partial A} u_A(x)  \geq (\leq) 0 \quad  \forall x > A. \label{equivalence_derivative_Phi}
\end{align}
If there exists $A^*$ such that 
\begin{align}
\Psi(A^*) - \frac {\sigma^2} 2 g'(A^*) = 0, \label{first_order_cond}
\end{align}
then the stopping time $\tau_{A^*}$ naturally becomes a reasonable candidate for the optimal stopping time.

%\red{This is satisfied for example (1) when $g$ is a bounded function or (2) when $g$ is Lipschitz continuous.}

In order to show Proposition \ref{lemma_derivative_A} above, we obtain the derivatives of $\Gamma_i$ for $1 \leq i \leq 3$ with respect to $A$ for any $x > A$.
By applying straightforward differentiation in Lemma \ref{lemma_gamma_1} and because $W^{(q)'}(x) = \Phi(q) W_{\Phi(q)} (x) + \Theta^{(q)}(x)$,
\begin{align}
\begin{split}
\frac \partial {\partial A}\Gamma_1(x;A) &=  g'(A) \Big[ Z^{(q)}(x-A) - \frac q {\Phi(q)} W^{(q)}(x-A) \Big] + g(A)  \frac q {\Phi(q)} \Theta^{(q)}(x-A).
\end{split} \label{derivative_gamma_1_A}
\end{align}
For $\Gamma_2$, we first take the derivatives of \eqref{gamma_2_rewrite} with respect to $A$.
%
%
%
%\begin{multline}
%\Gamma_2(x;A)  =   W_{\Phi(q)} (x-A) e^{\Phi(q) x}\int_0^\infty \Pi (\diff u)\int_A^{u+A} e^{- \Phi(q) y}   (g(y-u)-g(A)) \diff y  \\ -\int_0^\infty \Pi (\diff u) \left[  \int_A^{u+A} W^{(q)} (x-z) (g(z-u)-g(A)) \diff z \right], \label{gamma_2_rewrite}
%\end{multline}
%where the integral can be split thanks to Lemma \ref{lemma_lipschitz}. 

%In order to take the derivative with respect to $A$ inside the two integrals, we require that there exists $\delta > 0$ such that \eqref{integrability_sup} holds.

\begin{lemma} \label{lemma_derivative_A_individual}
Fix $x > A$.  Under the assumptions in  Proposition \ref{lemma_derivative_A},
\begin{multline} \label{derivative_A_individual_1}
\frac \partial {\partial A} \int_0^\infty \Pi (\diff u)\int_A^{u+A} e^{- \Phi(q) y}   (g(y-u)-g(A)) \diff y 
%&= \int_0^\infty \Pi (\diff u) \left[ e^{- \Phi(q) (u+A)}   g(A) - e^{-\Phi(q) A} g(A-u) - g(A) e^{-\Phi(q) (u+A)} + g(A) e^{-\Phi(q) A} - \frac {g'(A)} {\Phi(q)} \left( e^{-\Phi(q) A} - e^{-\Phi(q) (u+A)}\right)\right] 
%\\ &= \int_0^\infty \Pi (\diff u) \left[ (g(A)-g(A-u)) e^{-\Phi(q) A} - \frac {g'(A)} {\Phi(q)} \left( e^{-\Phi(q) A} - e^{-\Phi(q) (u+A)}\right)\right]
\\  =  e^{-\Phi(q) A} \int_0^\infty \Pi (\diff u) \Big[ g(A)-g(A-u)  - \frac {1 - e^{-\Phi(q) u}} {\Phi(q)} g'(A) \Big],
\end{multline}
and
\begin{align}
\frac \partial {\partial A}  \varphi_{g,A}^{(q)}(x)  =  \int_0^\infty \Pi (\diff u) \Big[  W^{(q)}(x-A) (g(A) - g(A-u)) - g'(A) \int_A^{(u+A) \wedge x}W^{(q)}(x-z) \diff z\Big].  \label{derivative_A_individual_2}
\end{align}
%\begin{enumerate}
%\item $\frac \partial {\partial A} \int_0^\infty \Pi (\diff u)\int_A^{u+A} e^{- \Phi(q) y}   [g(y-u)-g(A)] \diff y 
%%&= \int_0^\infty \Pi (\diff u) \left[ e^{- \Phi(q) (u+A)}   g(A) - e^{-\Phi(q) A} g(A-u) - g(A) e^{-\Phi(q) (u+A)} + g(A) e^{-\Phi(q) A} - \frac {g'(A)} {\Phi(q)} \left( e^{-\Phi(q) A} - e^{-\Phi(q) (u+A)}\right)\right] 
%%\\ &= \int_0^\infty \Pi (\diff u) \left[ (g(A)-g(A-u)) e^{-\Phi(q) A} - \frac {g'(A)} {\Phi(q)} \left( e^{-\Phi(q) A} - e^{-\Phi(q) (u+A)}\right)\right]
% =  e^{-\Phi(q) A} \int_0^\infty \Pi (\diff u) \left[ g(A)-g(A-u)  - \frac {g'(A)} {\Phi(q)} \left( 1 - e^{-\Phi(q) u}\right)\right],$
%%\begin{multline*}
%%\frac \partial {\partial A} \int_0^\infty \Pi (\diff u)\int_A^{u+A} e^{- \Phi(q) y}   [g(y-u)-g(A)] \diff y 
%%%&= \int_0^\infty \Pi (\diff u) \left[ e^{- \Phi(q) (u+A)}   g(A) - e^{-\Phi(q) A} g(A-u) - g(A) e^{-\Phi(q) (u+A)} + g(A) e^{-\Phi(q) A} - \frac {g'(A)} {\Phi(q)} \left( e^{-\Phi(q) A} - e^{-\Phi(q) (u+A)}\right)\right] 
%%%\\ &= \int_0^\infty \Pi (\diff u) \left[ (g(A)-g(A-u)) e^{-\Phi(q) A} - \frac {g'(A)} {\Phi(q)} \left( e^{-\Phi(q) A} - e^{-\Phi(q) (u+A)}\right)\right]
%% =  e^{-\Phi(q) A} \int_0^\infty \Pi (\diff u) \left[ g(A)-g(A-u)  - \frac {g'(A)} {\Phi(q)} \left( 1 - e^{-\Phi(q) u}\right)\right],
%%\end{multline*}
%\item $\frac \partial {\partial A}  \varphi_{g,A}^{(q)}(x)  =  \int_0^\infty \Pi (\diff u) \left[  W^{(q)}(x-A) (g(A) - g(A-u)) - g'(A) \int_A^{(u+A) \wedge x}W^{(q)}(x-z) \diff z\right].$
%\end{enumerate}
\end{lemma}
\begin{proof}
See Appendix \ref{proof_lemma_derivative_A_individual}.
\end{proof}

By applying Lemma \ref{lemma_derivative_A_individual} in \eqref{gamma_2_split}-\eqref{gamma_2_rewrite}, the derivative of $\Gamma_2$ with respect to $A$ is immediately obtained.
\begin{lemma}  \label{derivative_A_gamma_2}
Fix $x > A$.  Under the assumptions in  Proposition \ref{lemma_derivative_A},
\begin{multline*}
\frac \partial {\partial A}\Gamma_2(x;A)  =   -W_{\Phi(q)}' (x-A) e^{\Phi(q) x}\int_0^\infty \Pi (\diff u)\int_A^{u+A} e^{- \Phi(q) y}   (g(y-u)-g(A)) \diff y \\
+ g'(A) \int_0^\infty \Pi (\diff u) \Big( \int_A^{(u+A) \wedge x}W^{(q)}(x-z) \diff z -   \frac { 1 - e^{-\Phi(q) u}} {\Phi(q)} W^{(q)} (x-A) \Big).
\end{multline*}
\end{lemma}

For $\Gamma_3$, as in the proof of Lemma 4.4 of \cite{Egami-Yamazaki-2010-1}, we have the following.  Although the continuity of $h$ is assumed throughout in \cite{Egami-Yamazaki-2010-1}, it is not required in the following lemma; this is clear from the proof of Lemma 4.4 of \cite{Egami-Yamazaki-2010-1}.

% Here we assume for given $A \in \R$ that
%\[
%\frac \partial {\partial A}\int_0^\infty \zq h(y+A)\diff y = \int_0^\infty \zq h'(y+A)\diff y.
%\]
%By the mean value theorem, we can verify that this is satisfied for example when
%\begin{equation}\label{eq:h-assumption2}
%\int_0^\infty |\zq
%h'(y+A)| \diff y<\infty,
%\end{equation}
%there exists $\delta > 0$ such that
%\begin{align}
%\int_0^\infty e^{-\Phi(q) y} \sup_{0< \varepsilon < \delta}|
%h'(y+A+\varepsilon) - h'(y+A)| \diff y < \infty,
%\label{eq:h-assumption3}
%\end{align}
%Note that these are clearly satisfied if $h$ is a polynomial.

\begin{lemma}  \label{derivative_A_gamma_3}
For every $x > A$,
\begin{align*} \frac \partial
{\partial A} {\Gamma_3(x;A)} = - \Theta^{(q)}(x-A) 
\int_0^\infty e^{-\Phi(q) y} h(y+A) \diff y.
\end{align*}
\end{lemma}

We are now ready to prove  Proposition  \ref{lemma_derivative_A}.
\begin{proof}[Proof of Proposition  \ref{lemma_derivative_A}]
By combining \eqref{derivative_gamma_1_A} and Lemmas \ref{derivative_A_gamma_2} and \ref{derivative_A_gamma_3}, we
obtain
\begin{align*}
\frac \partial {\partial A} u_A(x)
%\\  &= g'(A) \left[ Z^{(q)}(x-A) - \frac q {\Phi(q)} W^{(q)}(x-A) \right] + g(A)  \frac q {\Phi(q)} e^{\Phi(q) (x-A)}W_{\Phi(q)}'(x-A) \\
% &-W_{\Phi(q)}' (x-A) e^{\Phi(q) (x-A)}\int_0^\infty \Pi (\diff u)\int_A^{u+A} e^{- \Phi(q) (y-A)}   (g(y-u)-g(A)) \diff y \\
%&+ g'(A) \int_0^\infty \Pi (\diff u) \left\{ \int_A^{(u+A) \wedge x}W^{(q)}(x-z) \diff z - W^{(q)} (x-A)  \frac 1 {\Phi(q)} \left( 1 - e^{-\Phi(q) u}\right)\right\} \\
%&- W_{\Phi(q)}'(x-A) e^{\Phi(q)(x-A)}
%\int_0^\infty e^{-\Phi(q) y} h(y+A) \\
= -\Theta^{(q)}(x-A)  \Psi(A) + g'(A)  Q(x;A)
%&= -W_{\Phi(q)}' (x-A) e^{\Phi(q) (x-A)} \Psi(A) + g'(A)Q(x;A). 
\end{align*}
where
\begin{multline*}
Q(x;A) := Z^{(q)}(x-A) - \frac q {\Phi(q)} W^{(q)}(x-A) \\ - \int_0^\infty \Pi (\diff u) \Big( W^{(q)} (x-A)  \frac  {1 - e^{-\Phi(q) u}} {\Phi(q)} - \int_A^{(u+A) \wedge x}W^{(q)}(x-z) \diff z \Big), \quad x > A.
\end{multline*}
By Lemma \ref{lemma_gamma_1} and modifying Lemma \ref{lemma_gamma_2},  we can also write
\begin{align*}
Q(x; A) = \E^x\left[ e^{-q \tau_A} \right] - \E^x\left[ e^{-q \tau_A} 1_{\{X_{\tau_A} < A, \, \tau_A < \infty\}}\right] = \E^x\left[ e^{-q \tau_A} 1_{\{X_{\tau_A} = A, \, \tau_A < \infty \}}\right], \quad x > A.
\end{align*}
A spectrally negative \lev process creeps downward if and only if there is a Gaussian component, i.e.,  
\begin{align*}
\p^x \left\{ X_{\tau_A} = A,  \tau_A < \infty \right\} > 0 \quad \forall x > A \Longleftrightarrow \sigma > 0;
\end{align*}
see Exercise 7.6 of \cite{Kyprianou_2006}.   Hence
\begin{align*}
\sigma > 0 \Longleftrightarrow Q(x;A) > 0, \; \forall x > A.
\end{align*}
This proves the desired result  for the case $\sigma = 0$.  For the case $\sigma > 0$, as in \cite{Biffis_Kyprianou_2010, Pistorius_2005}, we can also write
\begin{align*}
Q(x; A) = \frac {\sigma^2} 2 \left( W^{(q)'}(x-A) - \Phi(q) W^{(q)}(x-A)\right) = \frac {\sigma^2} 2  \Theta^{(q)}(x-A),
\end{align*}
and hence it also holds when $\sigma > 0$ as well.
\end{proof}

\subsection{Continuous and smooth fit}
We now pursue $A^*$ such that $u_{A^*}(A^*+) = g(A^*)$ and $u_{A^*}'(A^*+) = g'(A^*)$ for the cases
\begin{enumerate}
\item $X$ is of bounded variation, and
\item $\sigma > 0$,
%\item $g'(A^*) = 0$.
\end{enumerate}
respectively.  We exclude the case $X$ is of unbounded variation with $\sigma = 0$ (in this case, $W^{(q)'}(0+) = \infty$ by \eqref{at_zero} and hence the interchange of limits over integrals we conduct below may not be valid).   However, this can be alleviated and the results hold generally for all spectrally negative \lev processes when $g$ is a constant in a neighborhood of $A^*$. Examples include \cite{Egami-Yamazaki-2010-1} where $g(x) = 0$ on
$(0,\infty)$ and \cite{Salminen_1985} where $g(x) = 1$
on $(-\infty,0]$ and $g(x) = 2$ on $(0,\infty)$; see Section
\ref{section_extension}. 
For continuous fit, we need to obtain
\begin{align*}
\Gamma_1(A+; A) := \lim_{x \downarrow A}\Gamma_1(x; A), \quad \Gamma_2(A+; A) := \lim_{x \downarrow A}\Gamma_2(x ;A), \quad \textrm{and} \quad \Gamma_3(A+;A) := \lim_{x \downarrow A}\Gamma_3(x;A)
\end{align*}
if these limits exist.  Define also $\varphi_{g,A}^{(q)}(A+) := \lim_{x \downarrow A}\varphi_{g,A}^{(q)}(x)$,
if it exists.  It is easy to see that
\begin{align}
\Gamma_1(A+;A) =  g(A) \left( 1 - \frac  q {\Phi(q)} {W^{(q)} (0)}  \right) \quad \textrm{and} \quad 
\Gamma_3(A+;A) =
W^{(q)} (0) \int_0^\infty e^{- \Phi(q) y}   h(y+A) \diff y. \label{limit_gamma_1_3}
\end{align}
%For $\Gamma_2(A+;A)$, by Assumption \ref{assumption_fit}, we can split the integrals and

The result for $\Gamma_2$ is immediate by the dominated convergence theorem thanks to Lemma \ref{lemma_lipschitz} and \eqref{bound_rho_by_rho}-\eqref{gamma_2_split}.
\begin{lemma} \label{lemma_rho}
Given (1)-(2) of Lemma \ref{lemma_lipschitz} for a given $A \in \R$,  we have
\begin{enumerate}
%\item $\overline{\varphi}_{g,A}^{(q)}(x) \leq e^{\Phi(q) (x-A)} \frac {\overline{\rho}_{g,A}^{(q)}} {\psi'(\Phi(q))}$ for every $x > A.$
\item $\varphi_{g,A}^{(q)}(A+) = 0$,
%\begin{align*}
%\rho_{g,A} < \infty \Longrightarrow \varphi_{g,A}(A+) = 0 \quad \textrm{and} \quad \rho_A < \infty \Longrightarrow \varphi_A(A+) = 0.
%\end{align*}
\item  $\Gamma_2(A+; A) = W^{(q)}(0) \rho_{g,A}^{(q)}$.
\end{enumerate}
\end{lemma}
%\begin{proof}[Proof of Lemma \ref{lemma_rho}]
%(1) , we have
%
%
%(2) It can be obtained by taking limits thanks to (1).
%
%(3) This is immediate thanks to (2) and \eqref{gamma_2_split} and Lemma \ref{lemma_gamma_1}.
%\end{proof}
Now Lemma \ref{lemma_rho} and \eqref{limit_gamma_1_3} show
\begin{align}
u_A(A+) = g(A) +  W^{(q)}(0) \Psi(A). \label{cont_fit_difference}
\end{align}
This together with \eqref{at_zero} shows the following.
%In particular, when $X$ is of unbounded variation, $W^{(q)}(0)=0$ and hence $u_A(A+;A) = g(A)$.
%%\begin{align}\label{eq:Case2-uA}
%%u_A(A+;A) = \sum_{1 \leq i \leq 3}\Gamma_i(A+;A) = g(A).
%%\end{align}
%In summary, we have the following result regarding the continuous fit condition.

\begin{proposition}[Continuous Fit] \label{lemma_continuous_fit}
Fix $A \in \R$ and suppose (1)-(2) of Lemma \ref{lemma_lipschitz} hold. 
\begin{enumerate}
\item If $X$ is of bounded variation, the continuous fit condition $u_A(A+) = g(A)$ holds if and only if
\begin{align*}
\Psi(A) = 0. %\label{smooth_cont_condition}
\end{align*}
\item If $X$ is of unbounded variation (including the case $\sigma = 0$), it is automatically satisfied.
\end{enumerate}
\end{proposition}

%\begin{proof}%[Proof of Lemma \ref{lemma_continuous_fit}]
%For the unbounded variation case ($W^{(q)}(0)=0$), Lemma \ref{lemma_rho}-(3) implies
%%
%%made on the \lev measure, we can split the
%%integral and \eqref{gamma_1_tilde} becomes \eqref{gamma_tilde_split}.
%%By the monotone convergence theorem (because $W^{(q)}$ is monotone)
%%for the last term on the right-hand side, the limit as $x \downarrow
%%A$ exists and it becomes
%%\begin{align*}
%%\widetilde{\Gamma}_1(A+;A)/g(A) =  1- \frac q {\Phi(q)} W^{(q)}(0) -
%%\frac {W^{(q)} (0)}  {\Phi(q)}\int_0^\infty \Pi (\diff u)  \left( 1
%%- e^{-\Phi(q) u}\right)
%%\end{align*}
%%due to \eqref{eq:at-zero}.   Because $\sigma > 0$, we have
%%$W^{(q)}(0)=0$ and hence
%\begin{align}\label{eq:Gamma1-at-A}
%\Gamma_1(A+;A)= g(A).
%\end{align}
%
%
%
%\end{proof}

%\textbf{Smooth Fit.}
For the case $X$ is of unbounded variation with $\sigma > 0$, we shall pursue smooth
fit condition at $A \in \R$. The following lemma says in this case
that the derivative can go into the integral sign and we can
further interchange the limit.

%Unfortunately, it needs technical assumption due to differentiation.  Here we define the following conditions and we assume when needed.
%\begin{align*}
%&(Assmp.1) \quad  \varphi'(x,A) = \int_0^\infty \Pi (\diff u) \int_0^u W^{(q)'} (x-z-A) \diff z, \\
%&(Assmp.2)  \quad \varphi_g'(x,A) =\int_0^\infty \Pi (\diff u) \left[  \int_0^{u} W^{(q)'} (x-z-A) g(z+A-u) \diff z \right], \\
%&(Assmp.1') \quad \varphi'(A+,A) = 0, \\
%&(Assmp.2') \quad \varphi_g'(A+,A) =0.
%\end{align*}
%Namely,  (Assmp.1)  and (Assmp.2) imply that the differentiation can go into integral, and (Assmp.1')  and (Assmp.2') further assume that the limit can be interchanged.  See the appendix for sufficient conditions for these assumptions to hold.

\begin{lemma} \label{lemma_for_smooth_fit}
Fix $A \in \R$.
%\begin{enumerate}
%\item \red{We have
%\begin{align*}
%\varphi_A^{(q)'}(x) = \int_0^\infty \Pi (\diff u) \int_0^{u  \wedge (x-A)} W^{(q)'} (x-z-A) \diff z,  \quad x > A \quad \textrm{and} \quad \varphi_A^{(q)'}(A+) = 0.
%\end{align*}}
%There exists $K < \infty$ such that
%\begin{align}
%\kappa_A (x) := \int_0^\infty \Pi (\diff u) \int_0^u W^{(q)'} (x-z-A) \diff z < K e^{\Phi(q) (x-A)} \label{smooth_fit_bound_1}
%\end{align}
%for all $x > A$. Furthermore,
 If $\sigma > 0$ and suppose (1)-(2) of Lemma \ref{lemma_lipschitz} hold,
then
\begin{align}
\varphi_{g,A}^{(q)'}(x) = \int_0^\infty \Pi (\diff u) \int_0^{u
\wedge (x-A)} W^{(q)'} (x-z-A) [g (z+A -u) - g(A)] \diff z, \quad x > A,
\label{varrho_derivative_g}
\end{align}
and
\begin{align}
\varphi_{g,A}^{(q)'}(A+) = 0.
\label{varrho_derivative_convergence_g}
\end{align}
%then there exists $K_A < \infty$ such that
%\begin{align}
%\kappa_{g,A} (x) := \int_0^\infty \Pi (\diff u) \int_0^u W^{(q)'} (x-z-A) g (z+A -u) \diff z < K_A e^{\Phi(q) (x-A)} \label{smooth_fit_bound_2}
%\end{align}
%for all $x > A$. Furthermore,
%\begin{align*}
%\varphi_g'(x,A) &= \kappa_{g,A}(x), \\
%\varphi_g'(A+,A) &= 0.
%\end{align*}
%\end{enumerate}
\end{lemma}
\begin{proof}
  See Appendix \ref{proof_lemma_for_smooth_fit}.
\end{proof}

%\begin{remark}
%In the case of unbounded variation with $\sigma = 0$,  it is expected that \eqref{varrho_derivative_g} holds but \eqref{varrho_derivative_convergence_g} does not.  This is because $W^{(q)'}(0+) = \infty$ and the limit cannot go into the integral.
%\end{remark}

We are now ready to obtain $\Gamma_i'(A+; A) $ for $1 \leq i \leq 3$.

\begin{lemma} \label{lemma_smooth_fit}
Fix $A \in \R$. Suppose $\sigma > 0$ and (1)-(2) of Lemma \ref{lemma_lipschitz} hold.  Then,
\begin{enumerate}
\item $\Gamma_1'(A+,A) =  -  W^{(q)'}(0+)  g(A) q/ {\Phi(q)}$,
\item $\Gamma_2'(A+;A) = W^{(q)'} (0+)  \rho_{g,A}^{(q)}$,
\item $\Gamma_3'(A+;A)= W^{(q)'} (0+) \int_0^\infty e^{- \Phi(q) y}   h(y+A) \diff y$.
\end{enumerate}
\end{lemma}
\begin{proof}
(1) It is immediate by Lemma \ref{lemma_gamma_1}. 
%By \eqref{gamma_1_tilde_split} and Lemma
%\ref{lemma_for_smooth_fit} (1),
%\begin{align*}
%\Gamma_1'(x;A)/g(A) =  q W^{(q)}(x-A) - \frac q {\Phi(q)} W^{(q)'}(x-A) -W^{(q)'} (x-A)  \frac {\rho^{(q)}} {\Phi(q)}  + \varphi_A'(x).
%\end{align*}
%By taking $A \downarrow 0$ via Lemma \ref{lemma_for_smooth_fit} (1) and by \eqref{q_rewrite}, we have
%\begin{align*}
%\Gamma_1'(A+,A) =  - g(A)  W^{(q)'}(0+) \frac 1 {\Phi(q)} \left[ q + \rho^{(q)} \right]  = - g(A)  W^{(q)'}(0+) \left( \mu + \frac 1 2 \sigma^2 \Phi(q) \right).
%\end{align*}
(2) By \eqref{gamma_2_split},
\begin{align*}
\Gamma_2'(x;A)  =   W^{(q)'} (x-A) \rho_{g,A}^{(q)} - \varphi_{g,A}^{(q)'}(x), \quad x > A.
\end{align*}
By taking $x \downarrow A$ via \eqref{varrho_derivative_convergence_g}, we have the claim.
(3) We have
\begin{align*}
\Gamma_3'(A+;A)&= \lim_{x \downarrow A}\left[ W^{(q)'} (x-A)
\int_0^\infty e^{- \Phi(q) y}   h(y+A) \diff y -
\int_A^x W^{(q)'} (x-y) h(y) \diff y \right] \\
&= W^{(q)'} (0+) \int_0^\infty e^{- \Phi(q) y}   h(y+A) \diff y.
\end{align*}
\end{proof}

By the lemma above, we obtain
\begin{align*}
u_A'(A+) = W^{(q)'}(0+) \Psi(A)
\end{align*}
or equivalently,  by virtue of \eqref{at_zero}, the smooth fit condition at $A^*$ is equivalent to \eqref{first_order_cond}.

\begin{proposition}[Smooth Fit] \label{lemma_smooth_fit2}
Fix $A \in \R$. Suppose $\sigma > 0$ and (1)-(2) of Lemma \ref{lemma_lipschitz} hold. 
Then, the smooth fit condition $u_A'(A+) = g'(A)$ holds if and only if
%\begin{align*}
%\Psi(A) = g(A)  \left( \mu + \frac 1 2 \sigma^2 \Phi(q) \right) + \frac 1 2 \sigma^2 g'(A).
%\end{align*}
\begin{align*}
\Psi(A) = \frac {\sigma^2} 2 g'(A).
\end{align*}
%\item Suppose $g'(A) = 0$.  Then the smooth fit condition is equivalent to \eqref{smooth_cont_condition}.
%\end{enumerate}
\end{proposition}
%\begin{proof}
%This is a direct application of Lemma \ref{lemma_smooth_fit}. The smooth-fit principle reads
%$\Gamma_1'(A+;A) + \Gamma_2'(A+;A) + \Gamma_3'(A+;A) = g'(A)$, or
%\begin{align}\label{eq:smooth-case2}
%\Psi(A) - g(A)  \frac 1 {\Phi(q)} \left[ q + \rho^{(q)} \right] = \Psi(A) - g(A)  \left( \mu + \frac 1 2 \sigma^2 \Phi(q) \right) = g'(A) /
%W^{(q)'} (0+)  = \frac 1 2 \sigma^2 g'(A).
%\end{align}
%\end{proof}

We summarize the results obtained in Propositions
\ref{lemma_continuous_fit}-\ref{lemma_smooth_fit2} in Table
\ref{table:summary}.
\begin{table}[htbp]
\begin{tabular}{|c|c| c|}
  \hline
       & Continuous fit & Smooth fit\\
\hline
\noindent bounded var.\ & $\Psi(A) = 0$ & N/A \\
\hline
$\sigma > 0$    & Automatically satisfied  &
$\Psi(A) = {\sigma^2} g'(A)/2$ \\
%  \hline
%unbounded.\ var.,  $\sigma = 0$ and $g'(A)=0$   & Automatically satisfied  &
%$\Psi(A) = 0$ \\
  \hline
\end{tabular}
\caption{Summary of Continuous- and Smooth-fit Conditions.}
\label{table:summary}
\end{table}
%\blue{In the table, if you mean (i) and (ii) are for case 3, (i) should be ``bounded variation" and (ii) should be unbounded variation.}
%\section{Examples}\label{sec:example}
%\subsection{Equivalence between first-order condition and smooth/continuous fit}
%We now compare the results obtained in 
It is clear from Proposition \ref{lemma_derivative_A} and Table \ref{table:summary} that the first-order condition and the continuous/smooth fit condition are indeed equivalent.

\subsection{Obtaining optimal solutions} \label{subsection_obtaining}  There are a number of examples where the optimality of the threshold strategy can be derived directly from the structure of the problem.  See, e.g., \cite{Christensen_2011_2, peskir-shiryaev} for examples of sufficient optimality conditions of the threshold strategy.  In such cases, the problem reduces to solely computing $A^*$. 

For a general problem where the optimality of the threshold strategy is not proven,  one needs to show that the value function satisfies the variational inequality:
\begin{enumerate}
\item[(i)] $u_{A^*}(x) \geq g(x)$ for all $x \in \R$,
\item[(ii)] $(\mathcal{L}-q) u_{A^*}(x) + h(x) = 0$ for all $x \in (A^*,\infty)$,
\item [(iii)] $(\mathcal{L}-q) u_{A^*}(x) + h(x) < 0$ for all $x \in (-\infty, A^*)$;
\end{enumerate}
see e.g., \cite{sulem}.  Here $\mathcal{L}$ is the infinitesimal generator associated with
the process $X$ applied to a sufficiently smooth function $f$
\begin{align*}
\mathcal{L} f(x) &:= c f'(x) + \frac 1 2 \sigma^2 f''(x) + \int_0^\infty \left[ f(x-z) - f(x) +  f'(x) z 1_{\{0 < z < 1\}} \right] \Pi(\diff z).
%\mathcal{L} f(x) &= \mu f'(x) + \int_0^\infty \left[ f(x-z) - f(x)
%\right] \Pi(\diff z).
\end{align*}
%for the unbounded and bounded variation cases, respectively; see
%\eqref{def_mu} for the definition of $\mu$.  
As we shall show shortly below, the conditions (i)-(ii) can be obtained upon some conditions.  
The proof of condition (iii) unfortunately relies on the structure of the problem; in order to complement this, we give examples where the optimality over all stopping times holds in the next section. 

\begin{lemma} \label{lemma_about_i}
Suppose $A^* \in \R$ satisfies \eqref{first_order_cond}, $g \in C^2[A^*, \infty)$ and
\begin{align}
\Psi(A) - \frac {\sigma^2} 2 g'(A) > 0, \quad 
\quad A>A^*. \label{inequality_above_A_star}
\end{align}
Then (i) is satisfied.
\end{lemma}
\begin{proof}
Because $g(x) = u_{A^*}(x)$ on $(-\infty,A^*]$, we only need to show (i) on $(A^*, \infty)$. 
%It can be proved similarly to Lemma \ref{lemma_derivative_A} (1) and we have
%\begin{align*}
%\frac \partial {\partial A} u_A(x) =     -W_{\Phi(q)}' (x-A) e^{\Phi(q) (x-A)} \Psi(A).
%\end{align*}
%By \eqref{equivalence_derivative_Phi} and because $\Psi(A)$ is increasing, we have
%\begin{align}\label{eq:comparison}
%\frac \partial {\partial A} u_A(x)  < 0 \quad \Longleftrightarrow
%\Psi(A) > 0 \Longleftrightarrow A > A^*
%\end{align}
%by recalling that $\Psi(A^*)=0$.   
For any $x > A^*$, we obtain by \eqref{equivalence_derivative_Phi} and \eqref{cont_fit_difference} that
\begin{align*}
u_{A^*}(x) \geq \lim_{A \uparrow x}u_A(x) = g(x) + W^{(q)}(0) \Psi(x).
\end{align*}
For the unbounded variation case, because $W^{(q)}(0)=0$, the result is immediate.  For the bounded variation case (which necessarily means $\sigma = 0$), \eqref{inequality_above_A_star} implies $\Psi(x) > 0$ and hence the result is also immediate.

%Here the first inequality holds by \eqref{equivalence_derivative_Phi}; the first equality holds by \eqref{cont_fit_difference}.  For
%the second inequality (i.e., $W^{(q)}(0) \Phi(x) \geq 0$), in
%unbounded variation case, it is in fact an equality since
%$W^{(q)}(0)=0$ (and the continuous-fit always holds).  On the other
%hand, in the bounded variation case, we used \eqref{eq:comparison} by
%noting $x>A^*$. %(\red{Q: Should the last equality $0=g(x)$ be $\ge
%%g(x)$ since we just assume that $g$ is negative?})  \blue{[Here $x$ is positive and $g(x)=0$ by assumption.]}
\end{proof}

%\begin{lemma}
%For (1), in view of  Proposition \ref{lemma_derivative_A}, suppose it can be proved that
%\[
%\Psi(A) - \frac {\sigma^2} 2 g'(A) > 0 \quad \Longleftrightarrow
%\quad A>A^*.
%\]  
%\end{lemma}

%
% The
%usual procedure is first to find an optimal value $A^*$ and
%construct a candidate value function based on that value. Next, one
%needs to show the above quasi-variational inequalities. We deal with
%the first task in the next section.  

Regarding the condition (ii), 
%the stochastic processes
%\begin{align*}
% \left\{ e^{-q(t \wedge  \tau_{B}^+ \wedge \tau_0)} W^{(q)} (X_{t \wedge  \tau_{B}^+ \wedge \tau_0}); t \geq 0 \right\} \quad \textrm{and} \quad \left\{e^{-q(t \wedge  \tau_{B}^+\wedge \tau_0)} Z^{(q)} (X_{t \wedge \tau_{B}^+\wedge \tau_0}); t \geq 0 \right\}
%\end{align*}
%for any  $B < \infty$ are martingales (see, e.g., \cite{Biffis_Kyprianou_2010}), and therefore it is known that
%\begin{align}
%(\mathcal{L} - q) W^{(q)}(x)  =
%(\mathcal{L} - q) Z^{(q)}(x)  = 0, \quad x
%> 0. \label{w_generator_zero}
%\end{align}
%Furthermore,
integration by parts can be applied to obtain the following (see Section A.5 of \cite{Egami-Yamazaki-2010-1} for a complete proof).
\begin{lemma}[Egami and Yamazaki
\cite{Egami-Yamazaki-2010-1}] \label{lemma_w_h}
If $h$ is continuous on $(A^*, \infty)$, we have
\begin{align*}
(\mathcal{L} - q) \left[ \int_{A^*}^x W^{(q)}(x-y) h(y) \diff y \right] = h(x), \quad x
> A^*.
\end{align*}
\end{lemma}
By Lemma \ref{lemma_w_h}, we obtain the following.

\begin{proposition} \label{proposition_harmonic}
Suppose, on $(A^*, \infty)$, $f(x) := \E^x \left[ e^{-q \tau_{A^*}} g(X_{\tau_{A^*}})\right] = \Gamma_1(x;{A^*}) + \Gamma_2(x;{A^*})$ is $C^1$ (resp.\ $C^2$)  for the case $X$ is of bounded (resp.\ unbounded) variation, and $h$ is continuous.  Assume also when  $X$ is of unbounded variation with $\sigma = 0$ that $W^{(q)}$ is $C^2$ on $(0,\infty)$. Then  $(\mathcal{L}-q) u_{A^*}(x) + h(x) = 0$  for any $x > A^*$.
\end{proposition}
\begin{proof}
See Appendix \ref{proof_proposition_harmonic}.
\end{proof}

\section{Examples} \label{section_extension}
In this section, we give examples to illustrate how we can apply the results obtained in the previous sections.  We first consider, as a warm-up, a generalized version of the McKean optimal stopping problem with additional running rewards.  We then extend Egami and Yamazaki \cite{Egami-Yamazaki-2010-1} and obtain analytical solutions.  We also give a brief review of Surya and Yamazaki \cite{Surya_Yamazaki_2012} and Yamazaki \cite{Yamazaki_2012} where the main results of this paper are directly applied.

\subsection{The McKean optimal stopping}   \label{subsection_McKean}
The classical McKean optimal stopping problem, also known as the pricing of a perpetual American put option, reduces to \eqref{general_problem} with $g(x) = K - e^x$ and $h \equiv 0$.  Here, $e^X$ models the stock price and $K > 0$ is the strike price; the option holder chooses a time to exercise so as to maximize the expected payoff.  
%This classical problem is known to have an optimal stopping time of threshold type (see \cite{Mordecki_2002}).  
In particular, for the spectrally negative case, it has been shown by \cite{Avram_2002} that the optimal threshold level is given by (when $\psi(1) \neq q$)
\begin{align}
A^* = \log \Big( K \frac q {\Phi(q)} \frac {\Phi(q)-1} {q-\psi(1)}\Big). \label{optimal_level_american_option}
\end{align}

We consider a more general case where $h$ is any non-decreasing and continuous function and give a simple proof by directly using the results obtained in the previous sections.   Here we assume that $\psi(1) \neq q$ (or $\Phi(q) \neq1$); the case of $\psi(1) = q$ can be obtained by taking limits on the results described below (see \cite{Yamazaki_2012} for details).
%\begin{remark}
%When the threshold strategy is guaranteed to be optimal such as the pricing of American put options (e.g.\ Mordecki ), it is sufficient to obtain $A$ that makes the derivative of $u_A$ with respect to $A$ vanishes. 
%\end{remark}
Because
\begin{align*}
-g(A) \frac q {\Phi(q)} + \rho_{g,A}^{(q)} &= -\frac q {\Phi(q)}(K-e^A) + \int_0^\infty \Pi (\diff u)  \int_0^{u} e^{-\Phi(q) z} (e^A - e^{z+A-u})\diff z \\
%&= -\frac q {\Phi(q)}(K-e^A) + e^A \int_0^\infty \Pi (\diff u) \left[  \int_0^{u} e^{-\Phi(q) z} (1 - e^{(z-u)})\diff z \right] \\
%&= -\frac q {\Phi(q)}(K-e^A) + e^A \int_0^\infty \Pi (\diff u) \left[  \frac {1 - e^{-\Phi(q) u}} {\Phi(q)} - e^{-u}\int_0^{u} e^{-(\Phi(q)-1) z}  \diff z \right] \\
%&= -\frac q {\Phi(q)}(K-e^A) + e^A \int_0^\infty \Pi (\diff u) \Big(  \frac {1 - e^{-\Phi(q) u}} {\Phi(q)} - e^{-u} \frac {1 - e^{-(\Phi(q)-1) u}} {\Phi(q) - 1}\Big) \\
&= -\frac q {\Phi(q)}K + e^A \Big[ \frac q {\Phi(q)} + \int_0^\infty \Pi (\diff u) \Big(  \frac {1 - e^{-\Phi(q) u}} {\Phi(q)} - e^{-u} \frac {1 - e^{-(\Phi(q)-1) u}} {\Phi(q) - 1}\Big) \Big],
\end{align*}
we obtain
\begin{align}
\Psi(A) - \frac {\sigma^2} 2 g'(A) = -\frac q {\Phi(q)}K + \frac {e^A} {\Phi(q)} M_q + \int_0^\infty e^{-\Phi(q) y} h(y+A) \diff y \label{Psi_american}
\end{align}
where
\begin{align*}
M_q &:= q + \frac {\sigma^2} 2 \Phi(q) + \int_0^\infty \Pi (\diff u) \Big[  (1 - e^{-\Phi(q) u}) - e^{-u} (1 - e^{-(\Phi(q)-1) u}) \frac {\Phi(q)} {\Phi(q) -1}\Big].
\end{align*}
Here, by the change of measure, $M_q$ can be simplified.
\begin{lemma} \label{lemma_M_q}
We have $M_q = \frac {\Phi(q)} {\Phi(q) - 1} (q- \psi(1))$.
\end{lemma}
\begin{proof}%[Proof of Lemma \ref{lemma_M_q}]
By the definition of $\psi$ and $\Phi$, we rewrite $M_q$ as
\begin{align*}
\begin{split}
 & q + \frac {\sigma^2} 2 \Phi(q) + \int_0^\infty \Pi (\diff u) \left[  (1 - e^{-\Phi(q) u} - \Phi(q) u 1_{\{ u \in (0,1) \}})  - e^{-u} (1 - e^{-(\Phi(q)-1) u}) \frac {\Phi(q)} {\Phi(q) -1} + \Phi(q) u 1_{\{ u \in (0,1) \}}\right] \\
&= \Big( c - \int_0^1 u (e^{-u}-1) \Pi(\diff u) \Big) \Phi(q) + \frac {\sigma^2} 2 \Phi(q) (\Phi(q) + 1 ) \\ &\qquad - \frac {\Phi(q)} {\Phi(q) -1} \int_0^\infty \Pi (\diff u) e^{-u} \left( 1 - e^{-(\Phi(q)-1) u} + (\Phi(q)-1) u 1_{\{ u \in (0,1) \}} \right).
\end{split} %\label{M_q_rewrite}
\end{align*}
%\begin{align*}
%\psi(1)  = c
%+\frac{1}{2}\sigma^2 +\int_{( 0,\infty)}(e^{-x}-1+x 1_{\{0 <x<1\}})\,\Pi(\diff x), 
%\end{align*}
%and
Define, as the Laplace exponent of $X$ under $\p_1$ with the change of measure $ \left. \frac {\diff \p_1} {\diff \p}\right|_{\mathcal{F}_t} = \exp(X_t - \psi(1) t)$, $t \geq 0$,
\begin{align*}
\psi_1(\beta) :=\Big(  \sigma^2  + c - \int_0^1 u (e^{-u}-1) \Pi(\diff u)\Big) \beta
+\frac{1}{2}\sigma^2 \beta^2 +\int_0^\infty (e^{- \beta u}-1 + \beta u 1_{\{ u \in (0,1) \}}) e^{-u}\,\Pi(\diff u).
\end{align*}
Then, $\psi_1(\Phi(q)-1) = \psi(\Phi(q)) - \psi(1) = q - \psi(1)$; see page 215 of \cite{Kyprianou_2006}.
Hence simple algebra shows
%\begin{align*}
%q- \psi(1)  &=\left(  \sigma^2  + c - \int_{(0,1)} u (e^{-u}-1) \Pi(\diff u)\right)  (\Phi(q) - 1)
%+\frac{1}{2}\sigma^2  (\Phi(q) - 1)^2 \\ &+\int_{( 0,\infty)}(e^{-  (\Phi(q) - 1) u}-1 +(\Phi(q) - 1) u 1_{u \in (0,1)}) e^{-u}\,\Pi(\diff u)
%\end{align*}
%or
\begin{align*}
\frac {\Phi(q)} {\Phi(q) - 1} (q- \psi(1))  &= \frac {\Phi(q)} {\Phi(q) - 1} \psi_1(\Phi(q)-1) = M_q,
%=\left(  \sigma^2  + c - \int_{(0,1)} u (e^{-u}-1) \Pi(\diff u)\right)  \Phi(q) \\
%&+\frac{1}{2}\sigma^2  (\Phi(q) - 1) \Phi(q) + \frac {\Phi(q)} {\Phi(q) - 1}  \int_{( 0,\infty)}(e^{-  (\Phi(q) - 1) u}-1 + (\Phi(q) - 1) u 1_{\{u \in (0,1) \}} ) e^{-u}\,\Pi(\diff u) \\
%&=\left( c - \int_{-1}^0 u (e^{-u}-1) \Pi(\diff u)\right)  \Phi(q) \\
%&+\frac{1}{2}\sigma^2  (\Phi(q)^2 + \Phi(q) ) + \frac {\Phi(q)} {\Phi(q) - 1}  \int_{( 0,\infty)}(e^{-  (\Phi(q) - 1) u}-1 +  (\Phi(q) - 1)  u 1_{\{u \in (0,1)}\}) e^{-u}\,\Pi(\diff u).
\end{align*}
as desired.
\end{proof}
It is clear that $M_q > 0$ and hence \eqref{Psi_american} is monotonically increasing in $A$.  Recall also Assumption \ref{assump_h}. Therefore on condition that 
\begin{align}
\lim_{A \downarrow -\infty}\Big[\Psi(A) - \frac {\sigma^2} 2 g'(A) \Big] = -\frac q {\Phi(q)}K + \lim_{A \downarrow -\infty} \int_0^\infty e^{-\Phi(q) y} h(y+A) \diff y < 0,
\label{cond_american}
\end{align}
there exists a unique $A^*$ such that  \eqref{Psi_american} vanishes and by \eqref{equivalence_derivative_Phi}
\begin{align}
\frac \partial {\partial A} u_A(x)  \geq 0 \quad  \forall x > A \Longleftrightarrow A \leq A^*. \label{american_derivative}
\end{align}
This shows that $\tau_{A^*}$ is optimal among the set of all stopping times of threshold type.  Notice in a special case
 when $h \equiv 0$, the optimal threshold $A^*$ reduces to \eqref{optimal_level_american_option}.  Because the optimal stopping time is known to be of threshold type by \cite{Mordecki_2002}, $\tau_{A^*}$ is indeed the optimal stopping time.

We now show that it is indeed optimal over all stopping times even when $h$ is not zero.  This reduces to showing (iii) (in Section \ref{subsection_obtaining}) because (i) holds thanks to \eqref{american_derivative} and Lemma \ref{lemma_about_i} and (ii) thanks to Proposition \ref{proposition_harmonic} and the smoothness of the value function given in Proposition \ref{proposition_value_mckean} below.  For this special case of $g$, we can simplify as in Exercise 8.7 (ii) and Corollary 9.3  of \cite{Kyprianou_2006} for any $x, A \in \R$,
\begin{multline}
\E^x \left[ e^{-q \tau_A} (K-e^{X_{\tau_A}}) \right]  = K \Big(Z^{(q)}(x-A) - \frac q {\Phi(q)}  W^{(q)}(x-A) \Big) \\ -e^x \Big(Z_1^{(q-\psi(1))} (x-A) -\frac {q-\psi(1)} {\Phi(q)-1}  W_1^{(q-\psi(1))}(x-A) \Big) \label{u_A_special}
\end{multline}
where $W_1$ and $Z_1$ are  versions of $W$ and $Z$ associated with the measure $\p_1$ under the same change of measure as in the proof of Lemma \ref{lemma_M_q}.  Here, notice as in Lemmas 8.3 and 8.5 of  \cite{Kyprianou_2006},  for each $x > 0$, the functions $q \mapsto W^{(q)}(x)$ and $q \mapsto Z^{(q)}(x)$ can be analytically extended to $q \in \mathbb{C}$.  In particular,  by Lemma 8.4 of \cite{Kyprianou_2006},
\begin{align}
e^x W_1^{(q-\psi(1))}(x) = W^{(q)}(x), \quad x \geq 0. \label{scale_measure_change}
\end{align}

\begin{proposition}  \label{proposition_value_mckean} Suppose $h$ is non-decreasing and continuous and satisfies Assumption \ref{assump_h} and \eqref{cond_american}.  Then there exists a unique $A^*$ such that \eqref{Psi_american} vanishes.  Moreover, $\tau_{A^*}$ is an optimal stopping time and the optimal value function is given by 
\begin{align*}
u_{A^*}(x) = K Z^{(q)}(x-A^*) - e^x Z_1^{(q-\psi(1))} (x-A^*)  -
\int_{A^*}^x W^{(q)} (x-y) h(y) \diff y.
\end{align*}
\end{proposition}
\begin{proof}
By Lemma \ref{lemma_M_q} and the discussion above this proposition, there exists a unique $A^*$ such that
\begin{align}
0 = -q K + e^{A^*} \frac {\Phi(q)} {\Phi(q) - 1} (q- \psi(1))  + \Phi(q) \int_0^\infty e^{-\Phi(q) y} h(y+A^*) \diff y. \label{american_verification_2}
\end{align}
By \eqref{u_A_special}-\eqref{american_verification_2},
\begin{align*}
u_{A^*}(x) &=K Z^{(q)}(x-A^*)  -e^x \left(Z_1^{(q-\psi(1))} (x-A^*) -W_1^{(q-\psi(1))}(x-A^*) \frac {q-\psi(1)} {\Phi(q)-1}\right) \\
&\qquad - W^{(q)}(x-A^*) e^{A^*} \frac {q-\psi(1)} {\Phi(q)-1}- \int_{A^*}^x W^{(q)}(x-y) h(y) \diff y \\
&= K Z^{(q)}(x-A^*)  -e^x Z_1^{(q-\psi(1))} (x-A^*)  - \int_{A^*}^x W^{(q)}(x-y) h(y) \diff y.
\end{align*}
Notice that $u_{A^*}$ is $C^1$ (resp.\ $C^2$) on $\R \backslash \{A^*\}$ and $C^0$ (resp.\ $C^1$) at $A^*$ when $X$ is of bounded (resp.\ unbounded) variation; see the proof of Lemma 4.5 of \cite{Egami-Yamazaki-2010-1} for the transformation of the integral term. 

We shall show (iii).  By \eqref{laplace_exp},
%\begin{align*}
%\mathcal{L} f(x) &= c f'(x) + \frac 1 2 \sigma^2 f''(x) + \int_0^\infty \left[ f(x-z) - f(x) +  f'(x) z 1_{\{0 < z < 1\}} \right] \Pi(\diff z),
%\end{align*}
%we have
$\mathcal{L} g(x) = -e^x \left[ c + \frac 1 2 \sigma^2 + \int_0^\infty \left[ e^{-z} - 1 +  z 1_{\{0 < z < 1\}} \right] \Pi(\diff z) \right] = -e^x \psi(1)$,
and hence
\begin{align}
(\mathcal{L}-q) g(x) + h(x) = - qK + e^x (q - \psi(1)) + h(x). \label{american_verification_1}
\end{align}
Because $h$ is non-decreasing and $x < A^*$
\begin{align}
\Phi(q) \int_0^\infty e^{-\Phi(q) y} h(y+A^*) \diff y \geq \Phi(q) \int_0^\infty e^{-\Phi(q) y} h(x) \diff y = h(x). \label{american_verification_3}
\end{align}
It is also easy to see that
\begin{align}
e^{A^*}\frac {\Phi(q)} {\Phi(q) - 1} (q- \psi(1)) \geq e^x (q - \psi(1)).  \label{american_verification_4}
\end{align}
Indeed, for the case $q - \psi(1) > 0$, we must have $\Phi(q)-1 > 0$ and hence \eqref{american_verification_4} holds by $A^* > x$; for the case $q - \psi(1) < 0$, the left hand side is positive while the right hand side is negative in \eqref{american_verification_4}.
%Indeed we have $q- \psi(1) > 0 \Longleftrightarrow \Phi(q) - 1 > 0$. The above inequality holds trivially, when $q - \psi(1) > 0$.  For the case $q-\psi(1) < 0$, it is enough to show
%\begin{align*}
%\frac {\Phi(q)} {1-\Phi(q)} (\psi(1)-q) \geq -(\psi(1) - q) \quad \textrm{or}  \quad \frac {\Phi(q)} {1-\Phi(q)} \geq -1,
%\end{align*}
%which holds trivially because $0 < \Phi(q) < 1$.  
By \eqref{american_verification_2}-\eqref{american_verification_4}, (iii) holds.

 This together with (i) and (ii) shows the optimality using a standard technique of optimal stopping; see \cite{Yamazaki_2012} for the rest of the proof of optimality.

\end{proof}

%\begin{remark} \label{remark_fit}
%As in \cite{Egami-Yamazaki-2011}, the continuous fit always holds for such $A^*$ and furthemore smooth fit also holds when $\sigma > 0$.  In fact, for this special case of stopping payoff, as in Kyprianou the smooth fit holds whenever $X$ is of unbounded variation.
%\end{remark}

%Therefore 
%\begin{align*}
%\frac \partial {\partial A} u_A(x) = -W_{\Phi(q)}' (x-A) e^{\Phi(q) (x-A)} \left( \Psi(A) - \frac {\sigma^2} 2 g'(A) \right)
%\end{align*}
%and because $g'(A) = -e^A$
%\begin{align*}
% \Psi(A) - \frac {\sigma^2} 2 g'(A) > 0 \Longleftrightarrow  \frac q {\Phi(q)}K <  e^A M \Longleftrightarrow  A > \log \left( \frac {q K} {M} \right)
%\end{align*}

\subsection{Generalization of Egami and Yamazaki \cite{Egami-Yamazaki-2010-1} }
We now solve an extension to \cite{Egami-Yamazaki-2010-1}, where
% The results in the last section are based on the assumption that the jump part of the process is of bounded variation. However, we can solve for any arbitrary spectrally negative \lev process by taking advantage of the form of $g$.
we obtained an alarm system that determines when a bank needs to start enhancing its own capital ratio so as not to violate the capital adequacy requirements. Here $X$ models the bank's  net worth or equity capital allocated to its loan/credit business.  The problem is to strike the balance between minimizing the chance of violating the net capital requirement and the costs of premature undertaking (or the regret) measured, respectively, by
\begin{align*}
R_x^{(q)}(\tau) := \E^x \left[ e^{-q \theta} 1_{\{\tau \geq \theta, \theta < \infty \}}\right] \quad \textrm{and} \quad H_x^{(q,h)} (\tau) := \E^x \left[ 1_{\{\tau < \infty\}} \int_\tau^\theta e^{-q t} h(X_t) \diff t \right]
\end{align*}
where $h$ is positive, continuous and increasing, and
\begin{align*}
\theta := \inf \left\{ t \geq 0: X_t \leq 0 \right\}
\end{align*}
denotes the capital requirement violation time.
We want to obtain over the set of stopping times,
\begin{align}
\S := \left\{ \tau \; \textrm{stopping time}: \tau \leq \theta \; a.s. \right\}, \label{s_precautionary}
\end{align}
an optimal stopping time that minimizes the linear combination of the two costs described above:
\begin{align*}
U_x^{(q,h)}(\tau,\gamma) := R_x^{(q)}(\tau) + \gamma H_x^{(q,h)} (\tau),
\end{align*}
for some $\gamma > 0$.
By taking advantage of the property of $\S$, the problem can be reduced to obtaining
\begin{align*}
\inf_{\tau \in \S}\E^x \left[ e^{-q \tau} 1_{\{X_\tau \leq 0, \tau < \infty \}} + \int_\tau^\theta e^{-q t} h(X_t) \diff t \right] = -u(x) + \E^x \left[ \int_0^\theta e^{-q t} h(X_t) \diff t \right]
\end{align*}
with
\begin{align*}
u(x) := \sup_{\tau \in \S}\E^x \left[ - e^{-q \tau} 1_{\{X_\tau \leq 0, \tau < \infty \}} + \int_0^\tau e^{-q t} h(X_t) \diff t \right].
\end{align*}
In other words, the problem reduces to \eqref{general_problem} with
\begin{align*}
  g(x)=\begin{cases}
    0, & x>0,\\
    -1, & x\le 0,
  \end{cases}
\end{align*}
and a special set of stopping times defined in \eqref{s_precautionary}.  Egami and Yamazaki \cite{Egami-Yamazaki-2010-1} solved for double exponential jump diffusion \cite{Kou_Wang_2003} and for a general spectrally negative \lev process.

We shall consider its extension for a more general $g$ (or more general $R_x^{(q)}(\tau) := -\E^x \left[ e^{-q \theta} g(X_\theta) 1_{\{ \tau \geq \theta, \theta < \infty \}}\right]$) by assuming the following.
\begin{spec}\normalfont\label{assumption_g_h}
\begin{enumerate}
  \item $g$ is negative and increasing on $(-\infty,0]$ (and zero on
  $(0,\infty)$) and satisfies, for $A > 0$, the assumptions in Proposition \ref{integrability_sup};
  \item $h$ is positive, continuous and increasing and satisfies Assumption \ref{assump_h}.
\end{enumerate}
\end{spec}
The first assumption on $g$ means that the
penalty $|g(X_\theta)|$ increases as the overshoot $|X_\theta|$
increases.  The second assumption on $h$ is the same as in \cite{Egami-Yamazaki-2010-1}; if a bank has a higher capital value, then it naturally has better access to high quality assets.

In this problem, it can be conjectured that there exists a threshold level $A^*$ such that $\tau_{A^*}$ is optimal.  
%As for the results in Section \ref{sec:fit}, we use the results obtained in Section \ref{sec:problem}. 
%\red{The main difference with those in Section \ref{sec:fit}  is that thanks to $g(x)=0$ on the positive line, the integration is over $(A,\infty)$ and not over $(0,\infty)$ and hence we no longer need Assumption \ref{assump_bounded_variation}; the results hold even when $X^j$ has paths of unbounded variation.}  
Here we can rewrite \eqref{def_rho} for all $A > 0$
\begin{align*}
\rho_{g,A}^{(q)}&=\int_A^\infty \Pi (\diff u)  \int_0^{u-A} e^{- \Phi(q) y}
g(y+A-u)\diff y, \\
\overline{\rho}_{g,A}^{(q)}&=\int_A^\infty \Pi (\diff u)  \int_0^{u-A} e^{- \Phi(q) y}
|g(y+A-u)| \diff y.
\end{align*}
%Assumption \ref{assumption_fit} is now a  minor condition that holds unless $g(x)$ explodes very rapidly as $x \downarrow -\infty$.  The following can be proven along the same line as Lemmas \ref{lemma_continuous_fit} and \ref{lemma_smooth_fit}.
This avoids the integration of $\Pi$ in the neighborhood of zero and hence Lemma \ref{lemma_for_smooth_fit} also holds for the case of unbounded variation with $\sigma = 0$.  Now, as a special case of Propositions \ref{lemma_continuous_fit}-\ref{lemma_smooth_fit2} (noticing $g(A) = g'(A) = 0$ for all $A > 0$), we obtain the following.

\begin{lemma}[Continuous and Smooth Fit]
Suppose (1)-(2) of Lemma \ref{lemma_lipschitz} for a given $A > 0$.
\begin{description}
\item[continuous fit] If $X$ is of bounded variation, the continuous fit condition $u_A(A+) = 0$ holds if and only if
\begin{align}
\Psi(A) = 0. \label{condition_precautionary}
\end{align}
If $X$ is of unbounded variation, it is automatically satisfied.
\item[smooth fit]  If $X$ is of unbounded variation, the smooth fit condition $u_A'(A+) = 0$ holds if and only if
\eqref{condition_precautionary} holds.
\end{description}
\end{lemma}

\begin{table}[htbp]
\begin{tabular}{|c|c| c|}
  \hline
       & Continuous fit & Smooth fit\\
\hline
(i) bounded var.   & $\Psi(A)=0$ &  N/A\\
 (ii) unbounded var.  & Automatically satisfied  & $\Psi(A)=0$  \\
  \hline
\end{tabular}

\caption{Summary of Continuous- and Smooth-fit Conditions.}
\end{table}

%\subsection{continuous and smooth fit}
%Now fix $A > 0$.  Upon the condition that $g(x)=0$ on $(0,\infty)$ and that $g$ is bounded, we can split the integral in \eqref{gamma_2} and we obtain
%\begin{multline}
%\Gamma_2(x;A)  =   W^{(q)} (x-A) \int_A^\infty \Pi (\diff u) \int_0^{u-A} e^{- \Phi(q) y}   g(y+A-u) \diff y  \\
%-\int_A^\infty \Pi (\diff u) \left[  \int_0^{u-A} W^{(q)} (x-z-A) g(z+A-u) \diff z \right].
%\end{multline}
%Therefore, by the dominated convergence theorem on the last term, its limit as $x \downarrow A$ vanishes and hence
%\begin{align*}
%\Gamma_2(A+;A)  &=   W^{(q)} (0) \int_A^\infty \Pi (\diff u) \int_0^{u-A} e^{- \Phi(q) y}   g(y+A-u) \diff y.
%\end{align*}
%As in the last section, we have
%\begin{align*}
%\Gamma_3(A+;A) &=
%W^{(q)} (0) \int_0^\infty e^{- \Phi(q) y}   h(y+A) \diff y.
%\end{align*}
%
%Therefore, $u_A(A+)=g(A)=0$ holds automatically when it is of unbounded variation (because $W^{(q)}(0)=0$).  On the other hand, it requires for the bounded variation case ($W^{(q)}(0)>0$) that
Under Specification \ref{assumption_g_h}, there exists at most one
$A^* > 0$ that satisfies \eqref{condition_precautionary} because
\begin{multline}
\Psi'(A) = \int_0^\infty e^{- \Phi(q) y}   h'(y+A) \diff y \\ + \int_A^\infty \Pi (\diff u)  \int_0^{u-A} e^{- \Phi(q) y}   g'(y+A-u)\diff y - \int_A^\infty \Pi (\diff u)  e^{- \Phi(q) (u-A)}   g(0-) > 0. \label{monotonicity_precautionary}
\end{multline}
\\
\noindent \textbf{Verification of optimality: } We let $A^*$ be the
unique root of $\Psi(A)=0$ if it exists and set it zero otherwise.  The optimality over $\S$ holds under the following assumption.
\begin{assump}
$W^{(q)}$ is $C^2$ on $(0,\infty)$ for the case $X$ is of unbounded variation with $\sigma = 0$.
\end{assump}
%By Assumption \ref{proposition_harmonic},
As in the case of the McKean optimal stopping problem, we only need to show (iii) in Section \ref{subsection_obtaining} because (i) holds by \eqref{monotonicity_precautionary} and Lemma \ref{lemma_about_i} and (ii) by Proposition \ref{proposition_harmonic} and the assumption above.

\begin{lemma}
If $A^* > 0$, we have $(\mathcal{L}-q) g(x) + h(x) \leq 0 $ for every $x \in (0,A^*)$.
\end{lemma}
\begin{proof}
Because $g(x) = g'(x)=0$ for every $x > 0$,
\begin{align}
 (\mathcal{L}-q) g(x) + h(x) = \int_x^\infty \Pi(\diff u) g(x-u) + h(x), \quad x \in (0,A^*). \label{generator_egami_yamazaki}
\end{align}
We shall show that this is negative.
Because $A^* > 0$, we must have
\begin{align*}
\int_0^\infty e^{- \Phi(q) y}   h(y+A^*) \diff y + \int_{A^*}^\infty \Pi (\diff u)  \int_0^{u-A^*} e^{- \Phi(q) y}   g(y+A^*-u)\diff y = 0,
\end{align*}
and hence
\begin{align*}
0 &\geq h(x) \int_0^\infty e^{- \Phi(q) y}    \diff y + \int_{A^*}^\infty \Pi (\diff u)  g(x-u) \int_0^{u-{A^*}} e^{- \Phi(q) y}   \diff y \\
&\geq h(x) \int_0^\infty e^{- \Phi(q) y}    \diff y + \int_{A^*}^\infty \Pi (\diff u)  g(x-u) \int_0^\infty e^{- \Phi(q) y}   \diff y \\
&\geq h(x) \int_0^\infty e^{- \Phi(q) y}    \diff y + \int_x^\infty \Pi (\diff u)  g(x-u) \int_0^\infty e^{- \Phi(q) y}   \diff y \\
&= \Big( h(x) + \int_x^\infty \Pi (\diff u)  g(x-u) \Big) \int_0^\infty e^{- \Phi(q) y}    \diff y,
\end{align*}
where the first inequality holds because $g$ and $h$ are increasing
and $x < A^*$, the second holds because $g$ is non-positive, and the
third holds because $x < A^*$ and $g$ is non-positive.  This together with \eqref{generator_egami_yamazaki} shows the result.
\end{proof}

%The optimality of $\tau_{A^*}$ holds thanks to (a)-(b).  
Now the optimality holds by the martingale argument.  For the rest of the proof, we refer the reader to the proof of Proposition 4.1 in \cite{Egami-Yamazaki-2010-1}.
\begin{proposition}
 If $A^* > 0$, then $\tau_{A^*}$ is the optimal stopping time  and the value function is given by $u_{A^*}(x)$ for every $x > 0$.   If $A^*=0$, then the value function is given by $\lim_{A \downarrow 0}u_{A}(x)$ for every $x > 0$.
\end{proposition}

\subsection{Other Examples}  The results of this paper are applicable to a wide range of optimal stopping problems.  Here we give a brief review of two recent papers where these are used.
%
%Due to the generality of the results of this paper, 
%
%The main results of this paper have been directly used in Surya and Yamazaki \cite{Surya_Yamazaki_2012} and in Yamazaki \cite{Yamazaki_2012}.  

Surya and Yamazaki \cite{Surya_Yamazaki_2012} generalized the spectrally negative \lev model \cite{ Hilberink_Rogers_2002, Kyprianou_Surya_2007} of the optimal capital structure problem with endogenous bankruptcy, originally studied by \cite{Leland_1994, Leland_Toft_1996}.  The problem is to obtain an optimal bankruptcy level so as to maximize the company's equity value.  This problem gives a classical framework of solving the tradeoff between minimizing bankruptcy costs and maximizing tax benefits in debt financing.
Using the results of this paper (in particular,  Propositions \ref{lemma_derivative_A}, \ref{lemma_continuous_fit}  and \ref{lemma_smooth_fit2}), they have succeeded in incorporating the scale effects by allowing the values of bankruptcy costs and tax benefits dependent on the firm's asset value.  Their results reduce to those of \cite{Hilberink_Rogers_2002, Kyprianou_Surya_2007} in the original setting where the bankruptcy cost is a fixed fraction of the asset value and the tax benefit rate is constant.  Furthermore, a series of numerical results validate the optimality of their solutions and hence also the main results of this current paper.

%They confirmed the validity of the main results of this paper in two ways.  First, by assuming the constant assumption as in the original paper, the results reduce to that of \cite{ Hilberink_Rogers_2002, Kyprianou_Surya_2007} .  Second, by a series of numerical results, they have shown that their optimality is numerically valid.
%\begin{itemize}
%\item verify the results of this paper by matching for the special case.
%\item by numerical results, the optimality is consitent.
%\end{itemize}
%
%In the former, they have succeeded in generalizing the spectrally negative \lev model TODO and TODO where one needs to find the optimal bankruptcy level that maximizes a firm's equity value.
%
%The optimal capital structure model with endogenous bankruptcy was first studied by Leland \cite{Leland_1994} and Leland and Toft  \cite{Leland_Toft_1996}, and was later extended to the spectrally negative \lev model by Hilberink and Rogers \cite{ Hilberink_Rogers_2002} and Kyprianou and Surya \cite{ Kyprianou_Surya_2007}.  

Yamazaki \cite{Yamazaki_2012} considered a multiple-stopping version of the problem discussed in  Subsection \ref{subsection_McKean} and also its generalization where $g$ is a general decreasing and concave function.  The optimal strategy is given by an increasing sequence of stopping times of threshold type as in \eqref{def_tau_A}, and these can be computed recursively without intricate computation.  The numerical results confirm the optimality for the one-stage case (as discussed in  Subsection \ref{subsection_McKean} of this current paper) and also for the multiple-stage case.

\section{Concluding Remarks} \label{section_conclusion}
We have discussed the optimal stopping problem for spectrally negative \lev processes.  By expressing the expected payoff  via the scale function, we achieved the first-order condition as well as the continuous/smooth fit condition and showed their equivalence.  The results obtained here can be applied to a wide range of optimal stopping problems for spectrally negative \lev processes.  As examples, we gave a short proof for the perpetual American option pricing problem and solved an extension to Egami and Yamazaki \cite{Egami-Yamazaki-2010-1}.

A natural direction of future research is to generalize it to the two-sided barrier case.  In this case, it is easily conjectured that the additional smooth fit (or the first-order) condition at the upper barrier needs to be incorporated.  This will give two equations for the lower and upper barrier levels; the solutions naturally become the optimal threshold levels.
Applications include, e.g. American strangles as in \cite{Sheu_2013}. 

Another direction is to pursue similar results for optimal stopping games.  Typically, the equilibrium strategies are given by stopping times of threshold type as in \cite{Baurdoux2008,Baurdoux2009, Leung_Yamazaki_2011} .  Similarly to the results obtained in this paper, the expected payoff admits expressions in terms of the scale function and hence the first-order condition and the continuous/smooth fit can be obtained analytically.  %We conjecture that there is a relationship between the first-order condition and the continuous/smooth fit conditions.

Finally, the results can be extended to a general \lev process with both positive and negative jumps.    This can potentially be obtained in terms of the Wiener-Hopf factor alternatively to the scale function. 
%Finally, the results can be extended to a number of variants of optimal stopping such as optimal switching and  impulse control.

%\subsection{Non-zero Terminal Reward on the Positive Axis}
%We consider another example which has non-zero reward on the
%positive axis.  Accordingly, we consider a problem in {\bf{Case 2}}.
%\begin{spec}\normalfont
%\begin{enumerate}
%\item Fix $b\ge 0$ and set our terminal reward as
%\begin{align*}
%  g(x):=\begin{cases}
%    \sqrt{x+b}, & x\ge -b, \\
%    0, & x<-b.
%  \end{cases}
%\end{align*}
%\item $h$ is positive and increasing.
%\end{enumerate}
%\end{spec}
%Since the continuous fit is automatically there, we look at the
%smooth-fit condition \eqref{eq:smooth-case2}, which we reproduce
%here:
%\begin{equation}\label{eq:smooth-reproduced}
%\Psi(A) = g(A)  \frac 1 {\Phi(q)} \left[ q + \int_0^\infty \Pi
%(\diff u) \left[ 1 - e^{-\Phi(q) u}\right] \right] +
%\frac{\sigma^2}{2} g'(A).
%\end{equation}

\appendix

\section{Proofs}

\subsection{Proof of Lemma \ref{lemma_lipschitz}} \label{proof_lemma_lipschitz}
By the assumption (1) and Taylor expansion, we can take $0 < \epsilon < 1$ such that, for any $0 < z < u < \epsilon$ and $\varrho_{A,\epsilon} := \max_{0 \leq \xi \leq \epsilon}|g''(A-\xi)| < \infty$,
\begin{align}
|g (A -u+z) - g(A)| \leq (u-z) |g'(A)| + \frac 1 2 (u-z)^2 \varrho_{A,\epsilon} \leq u |g'(A)| + \frac 1 2 u^2 \varrho_{A,\epsilon}. \label{taylor}
\end{align}
Therefore, by \eqref{integrability_levy_measure},
\begin{align*}
\int_0^{\epsilon} \Pi (\diff u)  \int_0^{u} e^{-\Phi(q) z} |g(z+A-u)-g(A)| \diff z\leq \int_0^{\epsilon} \Pi (\diff u) \Big( u^2 |g'(A)| + \frac 1 2 u^3 \varrho_{A,\epsilon} \Big) < \infty.
\end{align*}
%By splitting the integral,
%\begin{align*}
%\rho_{g,A}^{(q)} &:= \int_0^1 \Pi (\diff u) \left[  \int_0^{u} e^{-\Phi(q) z} |g(z+A-u)-g(A)| \diff z \right] + \int_1^\infty \Pi (\diff u) \left[  \int_0^{u} e^{-\Phi(q) z} |g(z+A-u)-g(A)| \diff z \right].
%\end{align*}
%The Lipschitz continuous assumption implies that there exists $K$ such that $|g(A-y) - g(A)| \leq Ky$ for any $y \in [0,1]$ and hence
On the other hand, by \eqref{integrability_of_g}, 
\begin{align*}\int_{\epsilon}^\infty \Pi (\diff u)  \int_0^{u} e^{-\Phi(q) z} |g(z+A-u)-g(A)| \diff z  \leq \frac 1 {\Phi(q)}\int_{\epsilon}^\infty \Pi (\diff u) \max_{A-u \leq \zeta \leq A}|g(\zeta)-g(A)| < \infty.
\end{align*}
Combining the above, the proof is complete.

\subsection{Proof of Lemma \ref{lemma_derivative_A_individual}} \label{proof_lemma_derivative_A_individual}
\underline{Proof of \eqref{derivative_A_individual_1}}:  Define $\varrho(A) := \int_0^\infty \Pi (\diff u)q(A;u)$ with
\begin{align*}
q(A;u) := \int_A^{u+A} e^{- \Phi(q) y}   [g(y-u) - g(A)]\diff y, \quad u \geq 0.
\end{align*}
By assumption, we can choose $0 < \epsilon < 1$ such that $g$ is $C^2$ on $[A-\epsilon, A+ \epsilon]$.

We choose $0 < \delta < \epsilon$ that satisfies
\eqref{integrability_sup} and fix $0 < c < \delta$.  By the mean value theorem, there exists $\xi \in (0,c)$ such that
\begin{align*}
q'(A+\xi;u) = \frac {q(A+c; u)-q(A; u)} c.
\end{align*}
For every $z \in (A, A+ c)$, we have
\begin{align}
q'(z;u) 
%&=   - e^{-\Phi(q) A} g(A-u)  + g(A) e^{-\Phi(q) A} - \frac {g'(A)} {\Phi(q)} \left( e^{-\Phi(q) A} - e^{-\Phi(q) (u+A)}\right)
%\\ &=   (g(A)-g(A-u)) e^{-\Phi(q) A} - \frac {g'(A)} {\Phi(q)} \left( e^{-\Phi(q) A} - e^{-\Phi(q) (u+A)}\right)
 &=  e^{-\Phi(q) z}  \Big( g(z)-g(z-u)  -   \frac {1 - e^{-\Phi(q) u}} {\Phi(q)} {g'(z)} \Big). \label{q_prime_expand}
\end{align}
The Taylor expansion implies that, for every $0 < u < \delta$,
\begin{align*}
\frac {|q(A+c;u)-q(A;u)|} c = |q'(A+\xi;u) | &\leq e^{-\Phi(q) (A+\xi)}  \frac {u^2} 2 \left[  \max_{0 \leq  \zeta \leq u}|g''(A+\xi-\zeta)| + \Phi(q) |g'(A+\xi)| \right] \\ &\leq e^{-\Phi(q) A}  \frac {u^2} 2 \left[  \max_{-\delta \leq  y \leq \delta}|g''(A+y)| + \Phi(q) \max_{0 \leq  y \leq \delta} |g'(A+y)| \right],
\end{align*}
uniformly in $c \in (0, \delta)$.  The integral of the right hand side over $(0,\delta)$ with respect to $\Pi$ is, 
%\begin{align*}
%\frac {|q(A+c;u)-q(A;u)|} c \leq e^{-\Phi(q) A}  \frac {u^2} 2 \left[  \max_{-\delta \leq  y \leq \delta}|g''(A+y)| + \Phi(q) \max_{0 \leq  y \leq \delta} |g'(A+y)| \right].
%\end{align*}
%Hence uniformly in $c \in (0, \delta)$ 
by \eqref{integrability_levy_measure},
\begin{align*}
\frac 1 2 e^{-\Phi(q) A} \left[  \max_{-\delta \leq  y \leq \delta}|g''(A+y)| + \Phi(q) \max_{0 \leq  y \leq \delta} |g'(A+y)| \right] \int_0^{\delta} u^2 \Pi(\diff u)  < \infty.
\end{align*}
On the other hand, for $u > \delta$, by \eqref{q_prime_expand},
\begin{align*}
 \frac {|q(A+c; u)-q(A; u)|} c \leq e^{-\Phi(q) A}  \Big( \max_{0 \leq \zeta \leq \delta}|g(A+\zeta)-g(A+\zeta-u)|  + \max_{0 \leq \zeta \leq \delta} \frac {|g'(A+\zeta)|} {\Phi(q)}  \Big)
\end{align*}
whose integral over $(\delta, \infty)$ equals
\begin{multline*}
%\int_{\delta}^\infty \Pi(\diff u) \frac {|q(A+c; u)-q(A; u)|} c \\ 
e^{-\Phi(q) A}  \left(  \int_{\delta}^\infty \max_{0 \leq \zeta \leq \delta} \left| (g(A+\zeta)-g(A+\zeta-u) \right|  \Pi(\diff u) + \max_{0 \leq \zeta \leq \delta}  \frac {|g'(A+\zeta)|} {\Phi(q)}  \Pi(\delta, \infty) \right),
\end{multline*}
%with
%\begin{multline*}
%K 
%%&:= 
%%\int_0^\infty \Pi(\diff u) \left[ \sup_{0 < \xi < \delta \wedge \epsilon}\left| g(A+\xi) \right| \left( 1-e^{-\Phi(q) u}\right)  + \sup_{0 < \xi < \delta \wedge \epsilon} \left| g(A+\xi) - g(A+\xi-u) \right|\right] \\
%%&= \sup_{0 < \xi < \delta \wedge \epsilon}\left| g(A+\xi) \right| \rho^{(q)} \\
%:= \sup_{0 < \xi < \delta \wedge \epsilon} \left| \frac {g'(A+\xi)} {\Phi(q)} \right| \Pi(\delta \wedge \epsilon, \infty)
%+ \int_{\delta \wedge \epsilon}^\infty \Pi(\diff u)\sup_{0 < \xi < \delta \wedge \epsilon} \left| (g(A+\xi)-g(A+\xi-u) \right|.
%\end{multline*}
which is finite by \eqref{integrability_sup} and how $\delta$ is chosen.

%because the first term is bounded because $g$ is differentiable on $(A, A+ \delta)$ and the first term is finite 
%by \eqref{integrability_levy_measure};  
%\begin{align*}
%\int_{\delta \wedge \epsilon}^1 \Pi(\diff u)\sup_{0 < \xi < \delta
%\wedge \epsilon} \left| g(A+\xi) - g(A+\xi-u) \right| \leq  \sup_{A
%\leq x \leq A+ (\delta \wedge \epsilon), A-1 \leq y \leq A+ (\delta
%\wedge \epsilon)} \left| g(x) - g(y) \right|  \Pi(\delta \wedge
%\epsilon, 1) < \infty;
%\end{align*}
% by the assumed
%\eqref{integrability_sup}.  
%we have
%\begin{align*}
%\int_0^{\delta \wedge \epsilon} \Pi(\diff u)\sup_{0 < \xi < \delta \wedge \epsilon} \left| g(A+\xi) - g(A+\xi-u) \right| \leq \sup_{A-(\delta \wedge \epsilon) < y < A+(\delta \wedge \epsilon)} g'(y) \int_0^\epsilon u \Pi(\diff u)
%\end{align*}
%which is finite because $X^j$ is of bounded variation and $g$ is differentiable on $(A-\epsilon, A+\epsilon)$, and we also have

This allows us to apply the dominated convergence theorem, and we obtain
\begin{multline*}
\lim_{c \downarrow 0}\frac {\varrho(A+c)-\varrho(A)} c = \int_0^\infty \Pi(\diff u) \lim_{c \downarrow 0} \frac {q(A+c;u)-q(A;u)} c \\ = \int_0^\infty \Pi(\diff u) q'(A;u) 
= e^{-\Phi(q) A} \int_0^\infty \Pi (\diff u)   \Big( g(A)-g(A-u)  -  {g'(A)} \frac {1 - e^{-\Phi(q) A}} {\Phi(q)} \Big).
\end{multline*}
The proof for the left-derivative is similar, and this completes the proof of \eqref{derivative_A_individual_1}.
%\begin{align*}
%\lim_{c \downarrow 0}\frac {\varrho(A)-\varrho(A-c)} c = e^{-\Phi(q) A} \int_0^\infty \Pi (\diff u)   \left( g(A)-g(A-u)  -  {g'(A)} \frac {1 - e^{-\Phi(q) A}} {\Phi(q)} \right).
%\end{align*}
%This completes the proof.

\underline{Proof of \eqref{derivative_A_individual_2}}:  Define
\begin{align}
\widetilde{q}(z; u,x) := \int_z^{(u+z) \wedge x} W^{(q)}(x-y)   [g(y-u)-g(z)] \diff y, \quad z \in \R \; \textrm{and} \; u > 0. \label{q_tilde}
\end{align}
Then, by \eqref{gamma_2_rewrite}, we have $\varphi_{g,A}^{(q)}(x) = \int_0^\infty \Pi (\diff u)\widetilde{q}(A;u,x)$.
%We shall show
%\begin{align*}
%\lim_{c \downarrow 0}\frac {\widetilde{\varrho}(A+c)-\widetilde{\varrho}(A)} c = \int_0^\infty \Pi (\diff u) \lim_{c \downarrow 0}\frac {\widetilde{q}(A+c; u)-\widetilde{q}(A; u)} c.
%\end{align*}
We use the same $0 < \delta < \epsilon$ as in the proof of \eqref{derivative_A_individual_1} above and fix $c$ and $\varepsilon$
such that
\begin{align*}
0 < c < \delta \wedge \frac {x-A} 4 =: \varepsilon. 
%\quad \textrm{and} \quad 0  < \varepsilon := \Big[ \frac 1 2(x-A) - \Big( \delta \wedge \frac {x-A} 4\Big) \Big] \wedge \delta.
\end{align*}
%It is then clear that $0 < c < \varepsilon \leq \delta$. 
%\red{Q6: To ensure $\varepsilon
%>c$ and also $x-u-A-\xi>(x-A)/4$ in the next page, would it be
%better to define
%\[
%\varepsilon := \frac 1 2(x-A) + \left[\delta \wedge \frac {x-A}
%4\right]?
%\]
%} 
%\red{Egami-san, Do you mean
%\[
%\varepsilon := \frac 1 2(x-A) - \left[\delta \wedge \frac {x-A}
%4\right]?
%\]  I think both work in any case, but this may be cleaner.}
%We shall split
%\begin{multline}
%\int_0^\infty \Pi (\diff u)  \frac {|\widetilde{q}(A+c;u,x)-\widetilde{q}(A;u,x)|} c \\ = \int_0^\varepsilon \Pi (\diff u) \frac {|\widetilde{q}(A+c;u,x)-\widetilde{q}(A;u,x)|} c \ + \int_\varepsilon^\infty \Pi (\diff u) \frac {|\widetilde{q}(A+c;u,x)-\widetilde{q}(A;u,x)|} c, \label{eqn_split}
%\end{multline}
%and show that these two terms on the right-hand side are
%bounded in $c$ on $(0, \delta \wedge \frac {x-A} 4)$. 

%We first show that the first term is bounded in $c$. 
For every fixed $0 < u < \varepsilon$,  our assumptions imply that $\widetilde{q}(\cdot;u,x)$ is $C^2$ on $(A,A+c)$.
By the mean value theorem, there exists $\xi \in (0,c)$ such that
\begin{align} \label{q_prime_difference}
\widetilde{q}'(A+\xi;u,x) = \frac {\widetilde{q}(A+c;u,x)-\widetilde{q}(A;u,x)} c.
\end{align}
Given $z$ at which $g$ is differentiable and satisfies $u + z < x$, differentiating \eqref{q_tilde} gives
\begin{align*}
\widetilde{q}'(z; u,x) = W^{(q)}(x-z)   (g(z)-g(z-u))  - g'(z) \int_z^{u+z} W^{(q)}(x-y)   \diff y.
\end{align*}
Because $x-u-A-\xi > \frac {x-A} 4 > 0$ (and thus $u + (A+\xi) < x$) and $g$ is differentiable at $A+\xi$,
\begin{align*}
&\frac {|\widetilde{q}(A+c;u,x)-\widetilde{q}(A;u,x)|} c  = |\widetilde{q}'(A+\xi;u,x)| \\ &= \left| W^{(q)}(x-A-\xi) (g(A+\xi)-g(A+\xi-u))  - g'(A+\xi) \int_{A+\xi}^{u+A+\xi} W^{(q)}(x-y) \diff y \right| \\
%&\left| W^{(q)}(x-A-\xi) (g(A+\xi)-g(A+\xi-u))  - g'(A+\xi) \int_{A+\xi}^{u+A+\xi} W^{(q)}(x-y) \diff y \right| \\
&\leq W^{(q)}(x-A-\xi)  \left| g(A+\xi)-g(A+\xi-u) - u g'(A+\xi) \right| \\
&\qquad + \left| g'(A+\xi) \int_{A+\xi}^{u+A+\xi} (W^{(q)}(x-A-\xi)-W^{(q)}(x-y)) \diff y \right| \\
&\leq W^{(q)}(x-A) \left|  g(A+\xi)-g(A+\xi-u) - u g'(A+\xi) \right| \\&\qquad + u |g'(A+\xi)| \left|   W^{(q)}(x-A-\xi)-W^{(q)}(x-u-A-\xi)  \right| \\
%&\leq  |g(A+\xi)|  \left| W^{(q)}(x-u-A-\xi) - W^{(q)}(x-A-\xi) \right|  + W^{(q)}(x-u-A-\xi) |g(A+\xi)-g(A+\xi-u)|  \\
&\leq   f_1(A;u,x)  + f_2(A;u,x)
\end{align*}
where 
%\red{Q7: For the rest of the proof, would it be a bit more
%natural to have $\delta \wedge (x-A)/4$ for the upper bound of $\xi$
%since we have $0<c<\delta \wedge \frac {x-A}{4}$?} \red{Egami-san, I changed as follows.}
\begin{align*}
f_1(A;u,x) &:= W^{(q)}(x-A) \max_{0 \leq \zeta \leq \varepsilon}|g(A+\zeta)-g(A+\zeta-u) - u g'(A+\zeta)|, \\
f_2(A;u,x) &:= u \max_{0 \leq \zeta \leq \varepsilon}|g'(A+\zeta)| \max_{0 \leq \zeta \leq \varepsilon}  \left| W^{(q)}(x-A-\zeta) - W^{(q)}(x-u-A-\zeta)  \right|.
\end{align*}
%Moreover,
%\begin{align*}
%&\left| W^{(q)}(x-A-\xi) (g(A+\xi)-g(A+\xi-u))  - g'(A+\xi) \int_{A+\xi}^{u+A+\xi} W^{(q)}(x-y) \diff y \right| \\
%&\leq \left| W^{(q)}(x-A-\xi) (g(A+\xi)-g(A+\xi-u) - u g'(A+\xi))  \right| \\
%&\qquad + \left| g'(A+\xi) \int_{A+\xi}^{u+A+\xi} (W^{(q)}(x-A-\xi)-W^{(q)}(x-y)) \diff y \right| \\
%&\leq \left| W^{(q)}(x-A-\xi) (g(A+\xi)-g(A+\xi-u) - u g'(A+\xi)) \right| \\&\qquad + u |g'(A+\xi)| \left|   (W^{(q)}(x-A-\xi)-W^{(q)}(x-u-A-\xi))  \right|.
%\end{align*}
First, $\int_0^\varepsilon \Pi(\diff u) f_1 (A;u,x)$ is finite because, for every $u \leq \varepsilon$, we have $u \leq \delta$ and 
\begin{align*}
\max_{0 \leq \zeta \leq \varepsilon} |g(A+\zeta)-g(A+\zeta-u) - u g'(A+\zeta) | \leq  \frac {u^2} 2 \max_{A - \delta \leq \zeta \leq A+\delta}| g''(\zeta)|,
\end{align*}
which is $\Pi$-integrable over $(0,\varepsilon)$ by \eqref{integrability_levy_measure}.
  On the other hand, by \eqref{scale_function_asymptotic_version} and because $0 \leq \zeta \leq \varepsilon$ implies $x - u - A - \zeta > \frac {x-A} 4 > 0$, we have
\begin{align*}
&\left|  W^{(q)}(x-A-\zeta) - W^{(q)}(x-u-A-\zeta) \right|  \\ &= \left|  e^{\Phi(q) (x-A-\zeta)}W_{\Phi(q)}(x-A-\zeta) - e^{\Phi(q) (x-u-A-\zeta) }W_{\Phi(q)}(x-u-A-\zeta) \right| \\
&\leq \Big| \frac {e^{\Phi(q) (x-A-\zeta)} - e^{\Phi(q) (x-u-A-\zeta) } } {\psi'(\Phi(q))} \Big| + e^{\Phi(q) (x-u-A-\zeta)} \left| W_{\Phi(q)}(x-A-\zeta) - W_{\Phi(q)}(x-u-A-\zeta) \right| \\
&\leq e^{\Phi(q) (x-A)} \Big( \frac {1 - e^{-\Phi(q) u} } {\psi'(\Phi(q))}   +  u \max_{ \frac {x-A} 4 \leq y \leq x-A} W_{\Phi(q)}'(y) \Big),
\end{align*}
%Because
%\begin{align*}
%\left| W_{\Phi(q)}(x-A-\xi) - W_{\Phi(q)}(x-u-A-\xi) \right| \leq u \sup_{ \frac {x-A} 4 \leq y \leq x-A} W_{\Phi(q)}'(y),
%\end{align*}
and hence
\begin{align*}
&\int_0^\varepsilon \Pi (\diff u) f_2(A;u,x)  
%\\ &\leq \max_{0 \leq \zeta \leq \delta \wedge \frac {x-A} 4}|g'(A+\zeta)| \int_0^\varepsilon u \max_{0 \leq \zeta \leq \delta \wedge \frac {x-A} 4}  \left| W^{(q)}(x-A-\zeta) - W^{(q)}(x-u-A-\zeta) \right| \Pi(\diff u)  
\leq \max_{0 \leq \zeta \leq \varepsilon}|g'(A+\zeta)| e^{\Phi(q) (x-A)}  \int_0^\varepsilon  u \Big( \frac {1 - e^{-\Phi(q) u} } {\psi'(\Phi(q))}   +  u \max_{ \frac {x-A} 4 \leq y \leq x-A} W_{\Phi(q)}'(y)  \Big) \Pi(\diff u),
\end{align*}
which is finite by  \eqref{integrability_levy_measure}.  

We now fix $u
> \varepsilon$ (which implies $u
> c$).  We have
\begin{align*}
&\frac {|\widetilde{q}(A+c;u,x)-\widetilde{q}(A;u,x)|} c \leq B_1(A,c;u,x) + B_2(A,c;u,x)
\end{align*}
where
\begin{align*}
B_1(A,c;u,x) &:= \frac 1 c \left[ \int_{(u+A) \wedge x}^{(u+A+c) \wedge x} W^{(q)}(x-y) |g(y-u)-g(A+c)| \diff y + \int_A^{A+c} W^{(q)}(x-y) |g(y-u)-g(A)| \diff y \right],  \\
B_2(A,c;u,x)& := \frac {|g(A+c) - g(A)|} c \int_{A+c}^{(u+A) \wedge x} W^{(q)}(x-y)  \diff y.
\end{align*}
For the former, we have
\begin{align*}
%&= \frac 1 c \left[ \int_{A+c}^{(u+A+c) \wedge x} W^{(q)}(x-y)   (g(y-u)-g(A+c)) \diff y - \int_A^{(u+A) \wedge x} W^{(q)}(x-y)   (g(y-u)-g(A)) \diff y  \right]\\
%&\leq \frac 1 c \left[ \int_{(u+A) \wedge x}^{(u+A+c) \wedge x} W^{(q)}(x-y) |g(y-u)-g(A+c)| \diff y + \int_A^{A+c} W^{(q)}(x-y) |g(y-u)-g(A)| \diff y \right] \\
B_1(A,c;u,x) 
%&\leq W^{(q)}(x-A) \left[ 2 \sup_{A-u\leq z \leq A+c}|g(z)-g(A)|+ |g(A) - g(A+c)| \right]\\
&\leq 3W^{(q)}(x-A)  \max_{A-u\leq z \leq A+c}|g(z)-g(A)| \\
&\leq  3 W^{(q)}(x-A) \Big( \max_{A-u\leq z \leq A}|g(z)-g(A)| + \max_{0 \leq \zeta \leq \delta}|g(A+ \zeta) - g(A+\zeta-u)|\Big) =: \bar{B}_1 (A;u,x) .
\end{align*}
Here the first inequality holds because $|g(y-u)-g(A+c)| \leq |g(y-u)-g(A)| + |g(A)-g(A+c)|$ and, for $(u+A) \wedge x \leq y \leq (u+A+c) \wedge x$, it holds that $A-u \leq A \wedge (x-u) \leq y-u \leq (A+c) \wedge (x-u) \leq A+c$.  For the second inequality, it holds trivially when the maximum is attained for some $A-u \leq z \leq A$.  If it is attained at $z = A+l$ for some $0 < l \leq c$.  Then, because $A-u \leq A+l-u \leq A$ (thanks to $c < u$) and $c < \delta$
\begin{align*}
\max_{A-u\leq z \leq A+c}|g(z)-g(A)| &\leq |g(A+l-u)-g(A)|  +  |g(A+l) - g(A+l-u)| \\ &\leq \max_{A-u \leq z \leq A}|g(z)-g(A)|  + \max_{0 \leq \zeta \leq \delta}|g(A+ \zeta) - g(A+\zeta-u)|.
\end{align*}
%Hence
%\begin{align*}
%%&\leq  W^{(q)}(x-A) \left[ 2\sup_{A-u\leq z \leq A}|g(z)-g(A)| + 2 \sup_{0 \leq \xi \leq c}|g(A+ \xi) - g(A+\xi-u)|  + \sup_{0 \leq \xi \leq \delta} |g(A) - g(A+\delta)|\right] \\
%B_1(A,c,u)  &\leq  3 W^{(q)}(x-A) \left[ \sup_{A-u\leq z \leq A}|g(z)-g(A)| + \sup_{0 \leq \xi \leq \delta}|g(A+ \xi) - g(A+\xi-u)|\right].
%\end{align*}
For the latter, by the $C^2$ property of $g$ in the neighborhood of $A$, how $\delta$ is chosen and $c < \delta$, we obtain
\begin{align*}
B_2(A,c;u,x) &\leq  \frac {|g(A+c) - g(A)|} c  \int_{A+c}^x W^{(q)}(x-y) \diff y \\ &\leq \Big( |g'(A)| + \frac \delta 2  \max_{A \leq \zeta \leq A+\delta}|g''(\zeta)|\Big) \int_{A}^x e^{\Phi(q) (x-y)}W_{\Phi(q)}(x-y) \diff y \\
&\leq \frac 1 {\Phi(q) \psi'(\Phi(q))}\Big( |g'(A)| + \frac \delta 2  \max_{A \leq \zeta \leq A+\delta}|g''(\zeta)|\Big) e^{\Phi(q) (x-A)} =: \bar{B}_2(A;x).
\end{align*}
Combining these,
\begin{align*}
&\int_\varepsilon^\infty \Pi (\diff u) (\bar{B}_1(A;u,x)  + \bar{B}_2(A;x)) \\ &\leq  3 W^{(q)}(x-A)  \int_\varepsilon^\infty \Pi (\diff u)  \Big( \max_{A-u\leq z \leq A}|g(z)-g(A)| + \max_{0 \leq \zeta \leq \delta}|g(A+ \zeta) - g(A+\zeta-u)|\Big) \\
%W^{(q)}(x-A)  \int_\varepsilon^\infty \Pi(\diff u) \\ \times \left[ 2 \sup_{A-u\leq z \leq A}|g(z)-g(A)| + 2 \sup_{0 \leq \xi \leq \delta}|g(A+ \xi) - g(A+\xi-u)| + \sup_{0 \leq \xi \leq \delta} |g(A) - g(A+\xi)|\right],
&\qquad + \frac 1 {\Phi(q) \psi'(\Phi(q))}\Big( |g'(A)| + \frac \delta 2  \max_{A \leq \zeta \leq A+\delta}|g''(\zeta)|\Big) e^{\Phi(q) (x-A)} \Pi(\varepsilon, \infty),
\end{align*}
which is finite by \eqref{integrability_of_g} and \eqref{integrability_sup}.  

In summary, \eqref{q_prime_difference} is bounded uniformly in $c \in (0, \varepsilon)$ by a function which is $\Pi$-integrable.
Hence, by the dominated convergence theorem,
\begin{multline*}
\lim_{c \downarrow 0}\frac {\varphi_{g,A+c}^{(q)}(x) -\varphi_{g,A}^{(q)}(x)} c = \int_0^\infty \Pi(\diff u) \lim_{c \downarrow 0} \frac {\widetilde{q}(A+c;u,x)-\widetilde{q}(A;u,x)} c = \int_0^\infty \Pi(\diff u) \widetilde{q}'(A;u,x) \\
= \int_0^\infty \Pi (\diff u) \Big[  W^{(q)}(x-A) (g(A) - g(A-u)) - g'(A) \int_A^{(u+A) \wedge x}W^{(q)}(x-z) \diff z\Big].
\end{multline*}
The result for the left-derivative can be proved in the same way.

\subsection{Proof of Lemma \ref{lemma_for_smooth_fit}} \label{proof_lemma_for_smooth_fit}
It is known as in \cite{Chan_2009} that  $\sigma > 0$ guarantees that $W_{\Phi(q)}$ is twice
continuously differentiable and hence
$W_{\Phi(q)}'$ is continuous on $(0,\infty)$.  Furthermore,
\eqref{at_zero} implies $W_{\Phi(q)}'(0+) = \frac 2 {\sigma^2} < \infty$
and \eqref{scale_function_asymptotic_version} implies
$\lim_{x \uparrow \infty}W_{\Phi(q)}'(x)  = 0$.  Therefore, there exists $L < \infty$
such that
\begin{align*}
L := \sup_{x > 0}W_{\Phi(q)}'(x). %\label{bound_scale_derivative}
\end{align*}

%In view of Lemma \ref{bound_c_1} above, define
%\begin{align} \label{lemma_eqn_split}
%\begin{split}
%\kappa_{g,A}^{(q)}(x)  &:= \int_0^\infty \Pi (\diff u) \frac \partial {\partial x}\int_0^{u \wedge (x-A)} W^{(q)} (x-z-A) |g (z+A -u) - g(A)| \diff z  \\ &= \int_0^\infty \Pi (\diff u) \int_0^{u  \wedge (x-A)} W^{(q)'} (x-z-A) |g (z+A -u) - g(A)| \diff z  \\
%&= \phi_1(x,A, 0) + \phi_2(x,A, 0),
%\end{split}
%\end{align}
%which has a bound
%\begin{align}
%\kappa_{g,A}^{(q)} (x) \leq (C_{1,A} + C_{2,A}) e^{\Phi(q) (x-A)}. \label{smooth_fit_bound_1}
%\end{align}
%with
%\begin{align*}
%K_A :=  \max_{y \in[A-1, A]}|g(y)| \left[ L \int_0^1 u \Pi (\diff u)    +  \frac 1 {\psi'(\Phi(q))}\int_0^1 \Pi (\diff u)  \left( 1 - e^{-\Phi(q) u} \right) \right] + \frac 1 {\psi'(\Phi(q))} \int_1^\infty \Pi (\diff u)  \max_{y \in[A-u, A]}|g(y)|,
%\end{align*}
%which is finite by \eqref{integrability_of_g} and because $X^j$  is of bounded variation.

For every fixed $0 <  c < 1$ and $u > 0$,
\begin{multline}
\frac 1 c  \left|  \int_0^{u \wedge (x+c-A)} W^{(q)} (x+c-z-A) (g(z+A-u)-g(A)) \diff z \right. \\ \left.  - \int_0^{u \wedge (x-A)} W^{(q)} (x-z-A) (g(z+A-u)-g(A)) \diff z  \right| \leq  l_1(x,A,c, u) + l_2(x,A,c, u),  \label{ratio_c_g}
\end{multline}
%\begin{align}
%\left| \frac {\varphi_{g,A}^{(q)}(x+c) - \varphi_{g,A}^{(q)}(x)} c \right| \leq  f_1(x,A,c) + f_2(x,A,c),  \label{ratio_c_g}
%\end{align}
where
\begin{align*}
l_1(x,A,c, u) &:= \int_0^{u \wedge (x-A)} q(x,c,z,A)  |g(z+A-u)-g(A)|  \diff z, \\
l_2(x,A,c, u) &:=  \int_{u \wedge (x-A)}^{u \wedge (x+c-A)} \frac {W^{(q)}(x+c-z-A)} c  |g(z+A-u)-g(A)|  \diff z, \\
q(x,c,z,A) &:= \frac {W^{(q)} (x+c-z-A) - W^{(q)} (x-z-A)} c.
\end{align*}

Because
\begin{align*}
q(x,c,z,A) &= \frac {e^{\Phi(q) (x+c-z-A)}W_{\Phi(q)} (x+c-z-A) - e^{\Phi(q)(x-z-A)}W_{\Phi(q)} (x-z-A)} c \\
&= e^{\Phi(q) (x-z-A)} \frac { (e^{\Phi(q) c} - 1)W_{\Phi(q)} (x+c-z-A) + \left(W_{\Phi(q)} (x+c-z-A) - W_{\Phi(q)} (x-z-A) \right)} c \\
&\leq e^{\Phi(q) (x-z-A)} \sup_{0 < \delta < 1}\Big( \frac { e^{\Phi(q) \delta} - 1} {\delta \psi'(\Phi(q))} + L \Big),
\end{align*}
we have
\begin{align*}
l_1(x,A,c, u) &\leq  \sup_{0 < \delta < 1}\Big( \frac { e^{\Phi(q) \delta} - 1} {\delta \psi'(\Phi(q))} + L \Big) \int_0^{u \wedge (x-A)} e^{\Phi(q) (x-z-A)}  |g(z+A-u)-g(A)|  \diff z =:  \bar{l}_1(x,A, u),
\end{align*}
which is $\Pi$-integrable as
\begin{align*}
\int_{0}^\infty \Pi(\diff u)\bar{l}_1(x,A,u) &\leq e^{\Phi(q) (x-A)} \Big( \frac { e^{\Phi(q) c} - 1} {c \psi'(\Phi(q))} + L \Big)   \overline{\rho}^{(q)}_{g,A} < \infty.
\end{align*}
On the other hand,
\begin{align*}
l_2(x,A,c, u) &\leq  1_{\{ u > x-A \}}{W^{(q)}(x+1-A)}  \max_{A-u \leq y \leq A} |g(y)-g(A)| \\
%&\leq {W^{(q)}(x+c-A)} \int_{x-A}^\infty \Pi (\diff u)   \sup_{A-u \leq y \leq A} |g(y)-g(A)| \\
&\leq 1_{\{ u > x-A \}} \frac  {e^{\Phi(q) (x+1-A)}} {\psi'(\Phi(q))} \max_{A-u \leq y \leq A} |g(y)-g(A)| =: \bar{l}_2(x,A,u),
\end{align*}
%where the second inequality holds because $x>A$. 
which is also $\Pi$-integrable by \eqref{integrability_of_g}.

Now, by \eqref{ratio_c_g}, the dominated convergence theorem applies and noting $W^{(q)}(0)=0$,
\begin{align*}
\lim_{c \downarrow 0}\frac {\varphi_{g,A}^{(q)}(x+c) - \varphi_{g,A}^{(q)}(x)} c &=\int_0^\infty \Pi (\diff u) \frac \partial {\partial x}\int_0^{u \wedge (x-A)} W^{(q)} (x-z-A) (g (z+A -u) - g(A)) \diff z \\
&= \int_0^\infty \Pi (\diff u) \int_0^{u
\wedge (x-A)} W^{(q)'} (x-z-A) (g (z+A -u) - g(A)) \diff z.
\end{align*}
The left-derivative can be obtained in the same way.  This proves \eqref{varrho_derivative_g}.

For the proof of \eqref{varrho_derivative_convergence_g}, let
\begin{align*}
k(x,A, u) := \int_0^{u  \wedge (x-A)} W^{(q)'} (x-z-A) |g (z+A -u) - g(A)| \diff z, \quad x > A \; \textrm{and} \;  u > 0.
\end{align*}
Fix $A < x < A+1$ and choose $0 < \epsilon < 1$ and $\varrho_{A,\epsilon}$ as in the proof of Lemma \ref{lemma_lipschitz}. By \eqref{taylor}, for all $0 < u < \epsilon$,
\begin{align} \label{k_left_1}
\begin{split}
k(x,A, u) &\leq  \Big( u |g'(A)| + \frac 1 2 u^2 \varrho_{A,\epsilon}\Big) \int_0^{u  \wedge (x-A)} W^{(q)'} (x-z-A) \diff z \\
 &=  \Big( u |g'(A)| + \frac 1 2 u^2 \varrho_{A,\epsilon} \Big) \left( W^{(q)} (x-A) - W^{(q)} ((x-A-u) \vee 0) \right) \\
&\leq  e^{\Phi(q)}  \Big( u |g'(A)| + \frac 1 2 u^2 \varrho_{A,\epsilon} \Big)  \Big(Lu +  \frac {1 - e^{-\Phi(q) u}} {\psi'(\Phi(q))} \Big),
\end{split}
\end{align}
where the last inequality holds because, by \eqref{scale_function_asymptotic_version},
\begin{align*}
&W^{(q)} (x-A) - W^{(q)} ((x-A-u) \vee 0) \\ &=  e^{\Phi(q)(x-A)} \left[ \left( W_{\Phi(q)}(x-A) - W_{\Phi(q)}((x-A-u) \vee 0)\right) + \left( 1 - e^{-\Phi(q) (u \wedge (x-A))} \right) W_{\Phi(q)}((x-A-u) \vee 0) \right] \\ &\leq  e^{\Phi(q)(x-A)} \Big(Lu +  \frac {1 - e^{-\Phi(q) u}} {\psi'(\Phi(q))} \Big) \leq  e^{\Phi(q)}   \Big(Lu +  \frac {1 - e^{-\Phi(q) u}} {\psi'(\Phi(q))} \Big).
\end{align*}
On the other hand, for $u \geq \epsilon$,
\begin{align} \label{k_left_2}
\begin{split}
k(x,A, u) &\leq  \max_{A-u \leq y \leq A}|g(y) - g(A)|  \int_0^{u \wedge (x-A)} W^{(q)'} (x-z-A) \diff z \\ &\leq W^{(q)}(1) \max_{A-u \leq y \leq A}|g(y)-g(A)|.
\end{split}
%\\
%&= W_{\Phi(q)}(x-A) e^{\Phi(q)(x-A)}  \int_1^\infty \Pi (\diff u)  \max_{y \in[A-u, A]}|g(y)-g(A)|\\ &\leq \frac 1 {\psi'(\Phi(q))} e^{\Phi(q)(x-A)}   \int_1^\infty \Pi (\diff u)  \max_{y \in[A-u, A]}|g(y)-g(A)|.
\end{align}
If we define $\bar{k}(A, u)$ as the right hand sides of \eqref{k_left_1} and \eqref{k_left_2} for $0 < u < \epsilon$ and for $u \geq \epsilon$, respectively, then
\begin{align*}
\int_0^\infty \Pi(\diff u) \bar{k}(A, u) = &\int_0^\epsilon \Pi(\diff u)  e^{\Phi(q)}  \Big( u |g'(A)| + \frac 1 2 u^2 \varrho_{A,\epsilon} \Big)  \Big(Lu +  \frac {1 - e^{-\Phi(q) u}} {\psi'(\Phi(q))} \Big) \\
&+ \int_\epsilon^\infty \Pi(\diff u)  W^{(q)}(1) \max_{A-u \leq y \leq A}|g(y)-g(A)|,
\end{align*}
which is clearly finite by \eqref{integrability_levy_measure} and \eqref{integrability_of_g}. 

Now we can interchange the limit via the dominated convergence theorem as $x \downarrow A$ in \eqref{varrho_derivative_g} and obtain
\eqref{varrho_derivative_convergence_g}. This completes the proof.

\subsection{Proof of Proposition \ref{proposition_harmonic}} \label{proof_proposition_harmonic}
%\blue{[focus on the alternative proof because the original proof may be wrong.]}
For all $B > x > A^*$, we have by the strong Markov property,
\begin{align*}
\E^x \left[ e^{-q \tau_{A^*}} g(X_{\tau_{A^*}}) | \mathcal{F}_{t \wedge \tau_{A^*} \wedge \tau_B^+}\right] = e^{-q(t \wedge \tau_{A^*} \wedge \tau_B^+)} f(X_{t \wedge \tau_{A^*}\wedge \tau_B^+}).
\end{align*}
Taking expectation on both sides we obtain $f(x) = \E^x \left[ e^{-q \tau_{A^*}} g(X_{\tau_{A^*}}) \right] = \E^x \left[ e^{-q(t \wedge \tau_{A^*} \wedge \tau_B^+)} f(X_{t \wedge \tau_{A^*}\wedge \tau_B^+}) \right]$.
Hence $\left\{e^{-q(t \wedge \tau_{A^*}\wedge \tau_B^+)} f(X_{t \wedge \tau_{A^*} \wedge \tau_B^+}); t \geq 0 \right\}$ is a martingale and, because $B > A^*$ is arbitrary, $(\mathcal{L}-q) f(x) = (\mathcal{L}-q) (\Gamma_1(x;{A^*}) + \Gamma_2(x;{A^*}))  = 0$ on $({A^*},\infty)$; see also Section 4 of \cite{Biffis_Kyprianou_2010} for a more rigorous proof.

On the other hand, 
%the stochastic processes $\left\{ e^{-q(t \wedge  \tau_{B}^+ \wedge \tau_0)} W^{(q)} (X_{t \wedge  \tau_{B}^+ \wedge \tau_0}); t \geq 0 \right\} \quad \textrm{and} \quad \left\{e^{-q(t \wedge  \tau_{B}^+\wedge \tau_0)} Z^{(q)} (X_{t \wedge \tau_{B}^+\wedge \tau_0}); t \geq 0 \right\}
%\end{align*}
%for any  $B < \infty$ are martingales (see, e.g., \cite{Biffis_Kyprianou_2010}), and therefore 
it is known that
\begin{align}
(\mathcal{L} - q) W^{(q)}(x)  =
(\mathcal{L} - q) Z^{(q)}(x)  = 0, \quad x
> 0. \label{w_generator_zero}
\end{align}
%Using the fact that $(\mathcal{L}-q) Z^{(q)}(x) = (\mathcal{L}-q) W^{(q)}(x) = 0$ on $(0,\infty)$, we have
%\begin{multline*}
%(\mathcal{L} - q)\widetilde{\Gamma}_1(x;A) = - \frac {g(A) } q \int_0^\infty \Pi (\diff u) \left( (\mathcal{L}-q) Z^{(q)}(x-A-u)\right) \\
%= - \frac {g(A)} q   \int_{x-A}^\infty \Pi (\diff u) \left( (\mathcal{L}-q) 1 \right) = g(A) \Pi(x-A,\infty).
%\end{multline*}
%Similarly, we have
%\begin{multline*}
%(\mathcal{L} - q) \Gamma_2(x;A) = - \int_0^\infty \Pi (\diff u)
%(\mathcal{L} - q) \int_0^u W^{(q)} (x-z-A) g(z+A-u) \diff z \\
%= -\int_0^\infty \Pi (\diff u) g(x-u) 1_{\{u > x-A\}}\diff z = -\int_{x-A}^\infty \Pi (\diff u) g(x-u)
%\end{multline*}
%where the second equality holds b
This together with Lemma \ref{lemma_w_h} gives
\begin{align*}
(\mathcal{L} - q) \Gamma_3(x;{A^*}) = - (\mathcal{L} - q) \left[ \int_{A^*}^x
W^{(q)} (x-y) h(y) \diff y \right] = -h(x). %\label{harmonic_gamma_3}
\end{align*}
Summing up these, we have the claim.

\bibliographystyle{abbrv}
\bibliographystyle{apalike}

\bibliographystyle{agsm}
\small{\bibliography{ospbib}}

\end{document}